\RequirePackage{fix-cm}
\documentclass[smallextended]{svjour3}       
\smartqed  
\usepackage{xspace}
\usepackage{subfigure}
\usepackage{amsfonts}
\usepackage{amsmath}
\usepackage{enumerate}
\usepackage{framed}
\usepackage{graphicx}
\usepackage{lscape}
\usepackage{bm}
\usepackage{color}
\usepackage{multicol}
\usepackage{cite}

\newcommand{\Dt}{\Delta t}

\newcommand{\figref}[1]{{Figure~\ref{#1}}}

\newcommand{\secref}[1]{{Section~\ref{#1}}}
\renewcommand{\eqref}[1]{{(\ref{#1})}}

\newtheorem{assumption}[theorem]{Assumption}
\newcommand{\thmref}[1]{{Theorem~\ref{#1}}}
 \newcommand{\lemref}[1]{{Lemma~\ref{#1}}}
 \newcommand{\assref}[1]{{Assumption~\ref{#1}}}
\newcommand{\propref}[1]{{Proposition~\ref{#1}}}
\newcommand{\Shdt}{S_{h,\Dt}}

\usepackage{graphicx}
\begin{document}

\title{Strong convergence analysis of the stochastic exponential Rosenbrock scheme for the finite element discretization of  semilinear   SPDEs driven by  multiplicative  and additive noise}


\titlerunning{Stochastic exponential Rosenbrock method for  SPDEs.}        

\author{Jean Daniel Mukam,
    Antoine Tambue 
}

\authorrunning{J. D. Mukam, A. Tambue} 

\institute{ J. D. Mukam \at
              \\
            Technische Universit\"{a}t Chemnitz, 09126 Chemnitz, Germany.  \\
            Tel.: +49-15213471370 \\
          \email{ jean.d.mukam@aims-senegal.org \\
          \hspace*{1cm}  jean-daniel.mukam@s2015.tu-chemnitz.de}           
           \and
           A. Tambue (Corresponding author) \at
            The African Institute for Mathematical Sciences(AIMS) of South Africa and Stellenbosh University,\\
            Center for Research in Computational and Applied Mechanics (CERECAM), and Department of Mathematics and Applied Mathematics, University of Cape Town, 7701 Rondebosch, South Africa.\\
             Tel.: +27-785580321\\
   \email{antonio@aims.ac.za, tambuea@gmail.com} 
}
\date{Received: date / Accepted: date}

\maketitle

\begin{abstract}
In this paper, we consider  the numerical approximation of a general second order semilinear stochastic partial differential equation (SPDE) driven by multiplicative and additive noise.
Our main interest is on such SPDEs where the nonlinear part is stronger than the linear part also called stochastic reactive dominated transport equations.
Most numerical techniques, including current stochastic exponential integrators lose their good stability properties on such equations. 
Using finite element for space discretization, we propose a new scheme  appropriated on such equations, called stochastic exponential Rosenbrock scheme (SERS) based 
on local linearization at every time step of the semi-discrete equation obtained after space discretization.
We consider noise with finite trace  and give a strong convergence proof  of the new scheme toward the exact solution in the  root-mean-square  $L^2$ norm.
Numerical experiments to sustain theoretical results are provided.

  
\end{abstract}

\section{Introduction}
\label{intro}
The strong  numerical approximation of an It\^{o}
stochastic partial differential equation  defined in  the bounded  domain $ \Lambda \subset
\mathbb{R}^{d}\,(d=1,2,3)$ is analyzed. The domain $ \Lambda$ is assumed to be a convex polygon, or   has smooth boundary.
Boundary conditions on the domain $ \Lambda$
are typically Neumann, Dirichlet or Robin conditions.
More precisely, we consider in  the abstract setting the following stochastic partial differential equation
\begin{eqnarray}
\label{model}
dX(t)=[AX(t)+F(X(t))]dt+B(X(t))dW(t), \quad X(0)=X_0, \quad t\in[0,T],
\end{eqnarray}
on $H=L^2(\Lambda)$, $T>0$ is a final time, $F$ and $B$ are nonlinear functions,  $X_0$ is the initial data which  may be random,
$A$ is a linear operator, unbounded, not necessarily self adjoint, and the generator of an analytic semigroup $S(t):=e^{t A}, t\geq 0.$
The noise  $W(t)=W(x,t)$ is a $Q-$Wiener process defined in a filtered probability space $(\Omega,\mathcal{F}, \mathbb{P}, \{\mathcal{F}_t\}_{t\geq 0})$.
The filtration is assumed to fulfill the usual conditions (see \cite[Definition 2.1.11]{Prevot}).
We assume that the noise can be represented as 
\begin{eqnarray}
\label{noise}
W(x,t)=\sum_{i\in \mathbb{N}^d}\sqrt{q_i}e_i(x)\beta_i(t), \quad t\in[0,T],
\end{eqnarray} 
where $q_i$, $e_i$, $i\in\mathbb{N}^d$ are respectively the eigenvalues and the eigenfunctions of the covariance operator $Q$,
and $\beta_i$ are independent and identically distributed standard Brownian Motions. Precise assumptions on $F$, $B$, $X_0$ and $A$ 
will be given in the next section to ensure the existence of the unique mild solution $X$ of \eqref{model} which has the following representation (see \cite{Prato,Prevot}) for $t\in(0,T]$
\begin{eqnarray}
\label{mild1}
X(t)=S(t)X_0+\int_0^tS(t-s)F(X(s))ds+\int_0^tS(t-s)B(X(s))dW(s).
\end{eqnarray}
In few cases,  exact solutions are explicitly available, so  numerical techniques  are  the only tools to provide good approximations 
in more general cases (see for examples \cite{Antonio1,Printems,Raphael,Yan1,Yan2,Kovac1,Xiaojie1,Jentzen1}).
Approximations are done at two levels,  spatial approximation and  temporal approximation. For the spatial approximation, the finite difference, the 
finite element method and the Galerkin spectral method are usually used \cite{Antonio1,Raphael,Yan1,Yan2,Jentzen1,shardlow:2003}. 
As for PDEs, standard explicit time stepping methods for SPDEs are usually unstable for stiff problems and therefore  severe time step constraint is needed.
To overcome that drawback, standard implicit Euler  methods are usually used \cite{Printems,Raphael,Xiaojie2}.
Although standard  implicit Euler   methods \footnote{Full implicit or semi-implicit methods} are stable, their implementation 
requires significantly more computational effort, specially full implicit methods,
as Newton method is usually used to solve nonlinear algebraic equations at each time step. Recently, stochastic exponential integrators \cite{Antonio1,Xiaojie1,Jentzen1} 
  appeared as non standard explicit methods  efficient for SPDE \eqref{model}.
 All stochastic exponential integrators analyzed  in the literature for SPDEs \cite{Antonio1,Xiaojie1,Jentzen1} are bounded on the nonlinear problem as in \eqref{model}
 where the linear part $A$ and the nonlinear function $F$ are explicitly known a priori. Such approach is justified in situations where the nonlinear function $F$ is small. 
 Indeed when $F$ is small the linear operator $A$ drives the SPDE \eqref{model} and the good stability of the exponential integrators and semi-implicit method are ensured.
 In fact, in more realistic applications the nonlinear function $F$ 
 can be stronger. Typical examples are  semilinear  advection  diffusion reaction equations with stiff reaction term. In such  cases,  the SPDE \eqref{model}
 is driven by the nonlinear operator $F$ and both exponential integrators \cite{Antonio1,Xiaojie1,Jentzen1} and semi-implicit Euler \cite{Printems,Xiaojie2}
 will behave as explicit Euler-Maruyama scheme (see \secref{motivation}), therefore their good stability properties are lost.
 To overcome  this issue  we propose in  this work a new scheme called Stochastic Exponential Rosenbrock Scheme (SERS).
 Coupled with finite element  for space  discretization, the new scheme is based on a local linearization of the drift term at each time  step  in the  
 corresponding semi-discrete  problem of \eqref{model}. 
 The local linearization therefore weakens the nonlinear part of the drift such that  the linearized semi-discrete problem is driven by its linear part, which change at each time step.
 The standard stochastic exponential scheme \cite{Antonio1} is applied  at the end  to that linearized semi-discrete problem and the corresponding scheme is our new scheme.
 The challenge here is to deal with the new discrete semigroup which indeed is a semigroup  process, called stochastic perturbed semigroup.
 The key idea comes from the deterministic exponential Rosenbrock  method \cite{Alex1,Ostermann1,Ostermann2,Ramos,Antjd1}.
 Note that similar  schemes  for stochastic differential equations in finite dimensions have been proposed in \cite{Mora2,Mora1}.
 Using some  deterministic tools from \cite{Antjd1}, we propose a strong convergence proof of the new schemes where the linear operator $A$ is not necessarily self adjoint.
 Note that the orders of convergence  are the same with stochastic exponential schemes proposed in\cite{Antonio1}.  
 The deterministic part of  this scheme  is of order $2$ in time and  has been proven  to be efficient and robust in comparison to  standard schemes in many applications
 \cite{Advances, SebaGatam} where  the perturbed semigroup and related matrix functions
 have been computed using  the Krylov subspace technique \cite{kry} and fast Leja  points  technique \cite{LE1,Advances}. 
 For our new  stochastic scheme,  numerical simulations show  its good stability behavior  compared
 with a stochastic exponential scheme proposed  in \cite{Antonio1}, where the stochastic perturbed semigroup and related  matrix functions are computed  using Krylov subspace technique.\\
The rest of this paper is organized as follows. \secref{wellposed} is devoted to the mathematical setting, the numerical method and the main result. 
In  \secref{convergenceproof} some preparatory results and the proof of the main result are provided.  
In \secref{experiment} we provide some numerical experiments to sustain our theoretical results. We end  the paper in \secref{concludremark} by providing a concluding remark.

\section{Mathematical setting and main results}
\label{wellposed}
\subsection{Main assumptions and well posedness}
\label{notation}
Before we state the well posedness result,  let us define key functional spaces, norms and notations that  will be used in the rest of the paper. 
Let  $(H,\langle.,.\rangle_H,\Vert .\Vert)$ be a separable Hilbert space.  For all $p\geq 2$ and for a Banach space $U$,
we denote by $L^p(\Omega, U)$ the Banach space of all equivalence classes of $p$ integrable $U$-valued random variables. We denote  by $L(U,H)$ 
 the space of bounded linear mappings from $U$ to $H$ endowed with the usual  operator norm $\Vert .\Vert_{L(U,H)}$. By  $\mathcal{L}_2(U,H):=HS(U,H)$,
 we  denote the space of Hilbert-Schmidt operators from $U$ to $H$. 
  We equip $\mathcal{L}_2(U,H)$ with the norm
 \begin{eqnarray}
 \label{def1}
 \Vert l\Vert^2_{\mathcal{L}_2(U,H)} :=\sum_{i=1}^{\infty}\Vert l\psi_i\Vert^2,\quad l\in \mathcal{L}_2(U,H),
 \end{eqnarray}
 where $(\psi_i)_{i=1}^{\infty}$ is an orthonormal basis of $U$. Note that this definition is independent of the orthonormal basis of $U$.
  For simplicity, we use the notations $L(U,U)=:L(U)$ and $\mathcal{L}_2(U,U)=:\mathcal{L}_2(U)$. It is well known that for all $l\in L(U,H)$ and $l_1\in\mathcal{L}_2(U)$, $ll_1\in\mathcal{L}_2(U,H)$ and 
  \begin{eqnarray}
  \label{trace1}
  \Vert ll_1\Vert_{\mathcal{L}_2(U,H)}\leq \Vert l\Vert_{L(U,H)}\Vert l_1\Vert_{\mathcal{L}_2(U)}.
  \end{eqnarray}
 Throughout this paper $W(t)$ is a $Q$-wiener process. We assume that the covariance operator  $Q : H\longrightarrow H$ is positive and self-adjoint.  
 The space of Hilbert-Schmidt operators from  $Q^{1/2}(H)$ to $H$ is denoted by $L^0_2:=\mathcal{L}_2(Q^{1/2}(H),H)=HS(Q^{1/2}(H),H)$  
 with the corresponding norm $\Vert.\Vert_{L^0_2}$  defined by 
\begin{eqnarray*}
\Vert l\Vert_{L^0_2} :=\Vert lQ^{1/2}\Vert_{HS}=\left(\sum_{i=1}^{\infty}\Vert lQ^{1/2}e_i\Vert^2\right)^{1/2}, \quad  l\in L^0_2,
\end{eqnarray*} 
where $(e_i)_{i=1}^{\infty}$ is an orthonormal basis  of $H$.
This definition is independent of the orthonormal basis of $H$. In the rest of the paper, we take $H=L^2(\Lambda)$.\\
In order to ensure the existence and the uniqueness of solution of \eqref{model} and for the purpose of the convergence analysis, we make the following assumptions.
\begin{assumption}\textbf{[Linear operator $A$]}
\label{assumption1}
$A : \mathcal{D}(A)\subset H\longrightarrow H$ is a negative  generator of an analytic semigroup $S(t):=e^{At}$.
\end{assumption}
\begin{assumption}\textbf{[Initial value $X_0$]}
\label{assumption2}
We assume that $X_0\in L^p(\Omega, \mathcal{D}((-A)^{\beta/2}))$, $0\leq\beta\leq 2$, $p\geq 2$.
\end{assumption}
As in the current literature on deterministic exponential Rosenbrock-Type methods \cite{Julia1,Julia2,Antjd1,Alex1,Ostermann2}, we make the following assumption on the nonlinear term.
\begin{assumption}\textbf{[Nonlinear term $F$]}
\label{assumption3} 
We assume that the nonlinear mapping  $F: H\longrightarrow H$ is  Fr\'{e}chet differentiable  with bounded derivative, i.e. there exists a constant $C>0$ such that 
\begin{eqnarray*}
 \Vert F'(v)\Vert_{L(H)}\leq C, \quad \forall\, v\in H.
\end{eqnarray*}
\end{assumption}
Assumption \ref{assumption3} together with the mean value theorem show that there exists a constant  $L>0$ such that 
\begin{eqnarray}
\label{reviewinequal1}
\Vert F(Y)-F(Z)\Vert \leq L\Vert Y-Z\Vert, \quad Y,Z\in H.
\end{eqnarray}
As a consequence of \eqref{reviewinequal1}, there exists a positive constant $C$ such that
\begin{eqnarray*}
\Vert F(Z)\Vert&\leq& \Vert F(0)\Vert+\Vert F(Z)-F(0)\Vert\\
&\leq& \Vert F(0)\Vert+L\Vert Z\Vert\leq C (1+\Vert Z\Vert), \quad Z\in H.
\end{eqnarray*}
Following \cite[Chapter 7]{Prato} or \cite{Yan2,Antonio1,Arnulf1,Raphael} we make the following assumption on the diffusion term.
 \begin{assumption}\textbf{[Diffusion term ]} 
 \label{assumption4}
  We assume that the operator  $B : H \longrightarrow L^0_2$ satisfies the global Lipschitz condition, i.e. there exists a positive constant $C$ such that 
 \begin{eqnarray*}
 \Vert B(Y)-B(Z)\Vert_{L_2^0}\leq C\Vert Y-Z\Vert, \quad \forall\quad Y,Z\in H.
 \end{eqnarray*}
 As a consequence, there exists a positive constant $L>0$ such that 
 \begin{eqnarray*}
 \Vert B(Z)\Vert_{L^0_2}&\leq& \Vert B(0)\Vert_{L^0_2}+\Vert B(Z)-B(0)\Vert_{L^0_2}\nonumber\\
 &\leq& \Vert B(0)\Vert_{L^0_2}+C\Vert Z\Vert
 \leq L(1+\Vert Z\Vert), \quad \forall Z\in H.
 \end{eqnarray*}
 \end{assumption}
 To establish  our $L^2$ strong convergence result when dealing with multiplicative noise, we will also need the following further assumption on the diffusion term when $\beta \in [1,2)$, which was also used in  \cite{Antonio1,Arnulf1,Raphael,Stig1}.
\begin{assumption}
\label{assumption5}
We assume that there exist two positive constants $c>0$, and $\gamma\in(0,\frac{\beta}{10}]$  small enough such 
that $B(\mathcal{D}(-A)^{\gamma/2})\subset HS(Q^{1/2}(H),\mathcal{D}(-A)^{\gamma/2})$ and 
$\Vert (-A)^{\gamma/2}B(v)\Vert_{L^0_2}\leq c(1+\Vert (-A)^{\gamma/2}v\Vert)$ for all  $v\in\mathcal{D}((-A)^{\gamma/2})$, where $\beta$ is the parameter defined in Assumption \ref{assumption2}.
\end{assumption}
Typical examples satisfying \assref{assumption5} are stochastic reaction diffusion equations (see \cite[Section 4]{Arnulf1}).

 When dealing with additive noise, the strong convergence proof will make use of the following assumption on the noise.
 \begin{assumption}
 \label{assumption6a}
  We assume  that the covariance operator $Q : H\longrightarrow H$  satisfies the following estimate
 \begin{eqnarray}
 \left\Vert (-A)^{\frac{\beta-1}{2}}Q^{\frac{1}{2}}\right\Vert_{\mathcal{L}_2(H)}<\infty, 
 \end{eqnarray}
 where $\beta$ is defined in \assref{assumption2}.
 \end{assumption}
We equip $V_{\alpha}:=\mathcal{D}((-A)^{\alpha/2})$, $\alpha\in \mathbb{R}$ with the norm  $\Vert v\Vert_{\alpha}:=\Vert (-A)^{\alpha/2}v\Vert$, for all $v\in H$. 
It is well known that $(V_{\alpha}, \Vert .\Vert_{\alpha})$ is a Banach space \cite{Henry}. 

To achieve optimal order when dealing with additive noise, we require the nonlinear function $F$ to satisfy the following further assumption, also used in \cite{Antonio2,Xiaojie2,Xiaojie1,Xiaojie3}.
\begin{assumption}
\label{assumption6b}
We assume  that the deterministic mapping $F:   H \longrightarrow H$ is twice  differentiable and there exists a positive constant $C$ such that
\begin{eqnarray}
 \Vert F''(u)(v_1,v_2)\Vert_{-\eta}\leq C\Vert v_1\Vert.\Vert v_2\Vert, \quad u, v_1,v_2\in H,\quad\text{for some }\, \eta\in[1,2).
\end{eqnarray}

\end{assumption}

Let us recall in the following proposition  some  semigroup properties of the operator $S(t)$ generated by $A$ \footnote{The proposition 
indeed is general and provides some estimates for any semigroup and its generator.} that will be useful in the rest of the paper.  
 \begin{proposition}\textbf{[Smoothing properties of the semigroup]}\cite{Henry}
 \label{theorem1}
 \label{prop1} Let $\alpha > 0$, $\delta\geq 0$  and $0\leq \gamma\leq 1$, then there exists a constant $C>0$ such that 
 \begin{eqnarray*}
 \Vert (-A)^{\delta}S(t)\Vert_{L(H)}&\leq& Ct^{-\delta}, \quad t>0,\\
 \Vert (-A)^{-\gamma}(\mathbf{I}-S(t))\Vert_{L(H)}&\leq& Ct^{\gamma}, \quad t\geq 0,\\
 (-A)^{\delta}S(t)&=&S(t)(-A)^{\delta} \quad \text{on} \quad \mathcal{D}((-A)^{\delta}),\\
 \Vert D^l_tS(t)v\Vert_{\delta}&\leq& Ct^{-l-(\delta-\alpha)/2}\Vert v\Vert_{\alpha},\quad t>0,\quad v\in\mathcal{D}((-A)^{\alpha/2}),
 \end{eqnarray*}
 where $l=0,1$, 
 and  $D^l_t=\dfrac{d^l}{dt^l}$.\\
 If $\delta\geq \gamma$ then  $\mathcal{D}((-A)^{\delta})\subset \mathcal{D}((-A)^{\gamma})$.
 \end{proposition}

\begin{theorem}\textbf{[Well posedness result]}\cite[Theorem 7.4]{Prato}\\
\label{theorem2}
Let  \assref{assumption1}, \assref{assumption3} and \assref{assumption4} be satisfied. If $X_0$ is a $\mathcal{F}_0$- measurable $H$ valued random variable, 
then there exists a unique mild solution $X$ of problem \eqref{model} of the form \eqref{mild1} and  satisfying the following 
\begin{eqnarray*}
\mathbb{P}\left[\int_0^T\Vert X(s)\Vert^2ds<\infty\right]=1,
\end{eqnarray*}
and for any $p\geq 2$ there exists a constant $C=C(p,T)>0$ such that 
\begin{eqnarray*}
\sup_{t\in[0,T]}\mathbb{E}\Vert X(t)\Vert^p\leq C(1+\mathbb{E}\Vert X_0\Vert^p).
\end{eqnarray*}
\end{theorem}
Furthermore   from \cite[Theorem 1]{Arnulf1} or \cite[Theorem 2.6]{Antonio1} it holds that for all $\gamma\in[0,1)$, 
for all $p\geq 2$ there exists a positive constant $C$ such that 
\begin{eqnarray}
\label{regular1a}
\left(\mathbb{E}\Vert X(t)\Vert^p_{\gamma}\right)^{1/p}\leq C\left(1+(\mathbb{E}\Vert X_0\Vert^p_{\gamma})^{1/p}\right),\quad t\in[0,T].
\end{eqnarray}
\subsection{Finite element discretization}
\label{fullerror}
In the rest of this paper, to simplify the presentation, we assume that the linear operator $A$ is of second order. More precisely, 
we assume that our SPDE \eqref{model} is a second order semilinear parabolic and takes the form
\begin{eqnarray}
\label{secondorder}
dX(t,x)&=&[\nabla \cdot \left(\mathbf{D}\nabla X(t,x)\right)-\mathbf{q} \cdot \nabla X(t,x)+f(x,X(t,x))]dt\nonumber\\
&+&b(x,X(t,x))dW(t,x), \quad x\in\Lambda, \quad t\in[0,T],
\end{eqnarray}
where the functions $f : \Lambda\times \mathbb{R}\longrightarrow \mathbb{R}$ and
$b : \Lambda\times\mathbb{R}\longrightarrow \mathbb{R}$ are continuously differentiable with globally bounded derivatives.
In the abstract framework \eqref{model}, the linear operator $A$ takes the form
\begin{eqnarray}
\label{operator}
Au&=&\sum_{i,j=1}^{d}\dfrac{\partial}{\partial x_i}\left(D_{ij}(x)\dfrac{\partial u}{\partial x_j}\right)-\sum_{i=1}^dq_i(x)\dfrac{\partial u}{\partial x_i},\\
\mathbf{D}&=&\left(D_{i,j} \right)_{1\leq i,j \leq d}\,\,\,\,\,\,\, \mathbf{q}=\left( q_i \right)_{1 \leq i \leq d}.
\end{eqnarray}
where $D_{ij}\in L^{\infty}(\Lambda)$, $q_i\in L^{\infty}(\Lambda)$. We assume that there is a positive constant $c_1>0$ such that 
\begin{eqnarray*}
\sum_{i,j=1}^dD_{ij}(x)\xi_i\xi_j\geq c_1|\xi|^2, \quad \forall \xi\in \mathbb{R}^d,\quad x\in\overline{\Omega}.
\end{eqnarray*}
The functions $F : H\longrightarrow H$ and $B : H\longrightarrow HS(Q^{1/2}(H), H)$ are defined by 
\begin{eqnarray}
\label{nemystskii}
(F(v))(x)=f(x,v(x)) \quad \text{and} \quad (B(v)u)(x)=b(x,v(x)).u(x),
\end{eqnarray}
for all $x\in \Lambda$, $v\in H$, $u\in Q^{1/2}(H)$, with $H=L^2(\Lambda)$.
 For an appropriate family of eigenfunctions $(e_i)$ such that $\sup\limits_{i\in\mathbb{N}^d}\left[\sup\limits_{x\in \Lambda}\Vert e_i(x)\Vert\right]<\infty$, 
 it is well known \cite[Section 4]{Arnulf1} that the Nemystskii operator $F$ related to $f$ and the multiplication operator $B$ associated 
 to $b$ defined in \eqref{nemystskii} satisfy Assumption \ref{assumption3}, Assumption \ref{assumption4} and Assumption \ref{assumption5}.
As in \cite{Antonio1,Suzuki}, we introduce two spaces $\mathbb{H}$ and $V$, such that $\mathbb{H}\subset V$; the two spaces depend on the boundary 
conditions of $\Lambda$ and the domain of the operator $A$. For  Dirichlet (or first-type) boundary conditions we take 
\begin{eqnarray*}
V=\mathbb{H}=H^1_0(\Lambda)=\{v\in H^1(\Lambda) : v=0\quad \text{on}\quad \partial \Lambda\}.
\end{eqnarray*}
For Robin (third-type) boundary conditions and  Neumann (second-type) boundary condition, which is a special case of Robin boundary conditions, we take $V=H^1(\Lambda)$ and
\begin{eqnarray*}
\mathbb{H}=\{v\in H^2(\Lambda) : \partial v/\partial \mathtt{v}_{ A}+\alpha_0v=0,\quad \text{on}\quad \partial \Lambda\}, \quad \alpha_0\in\mathbb{R},
\end{eqnarray*}
where $\partial v/\partial \mathtt{v}_{ A}$ is the normal derivative of $v$ and $\mathtt{v}_{ A}$ is the exterior pointing normal $n=(n_i)$ to the boundary of $\Lambda$ given by
\begin{eqnarray*}
\partial v/\partial\mathtt{v}_{A}=\sum_{i,j=1}^dn_i(x)D_{ij}(x)\dfrac{\partial v}{\partial x_j},\,\,\qquad x \in \partial \Lambda.
\end{eqnarray*}
Using the Green's formula and the boundary conditions, the  corresponding bilinear form associated to $-A$  is given by
\begin{eqnarray*}
a(u,v)=\int_{\Lambda}\left(\sum_{i,j=1}^dD_{ij}\dfrac{\partial u}{\partial x_i}\dfrac{\partial v}{\partial x_j}+\sum_{i=1}^dq_i\dfrac{\partial u}{\partial x_i}v\right)dx, \quad u,v\in V,
\end{eqnarray*}
for Dirichlet and Neumann boundary conditions, and  
\begin{eqnarray*}
a(u,v)=\int_{\Lambda}\left(\sum_{i,j=1}^dD_{ij}\dfrac{\partial u}{\partial x_i}\dfrac{\partial v}{\partial x_j}+\sum_{i=1}^dq_i\dfrac{\partial u}{\partial x_i}v\right)dx+\int_{\partial\Lambda}\alpha_0uvdx, \quad u,v\in V.
\end{eqnarray*}
for Robin boundary conditions. Using the G\aa rding's inequality (\cite{ATthesis}), it holds that there exist two positive constants $c_0$ and $\lambda_0$ such that
\begin{eqnarray}
\label{ellip1}
a(v,v)\geq \lambda_0\Vert v \Vert^2_{H^1(\Lambda)}-c_0\Vert v\Vert^2, \quad \forall v\in V.
\end{eqnarray}
By adding and substracting $c_0Xdt$ on the right hand side of \eqref{model}, we have a new linear operator
that we still call $A$ corresponding to the new bilinear form that we still call $a$ such that the following coercivity property holds
\begin{eqnarray}
\label{ellip2}
a(v,v)\geq \lambda_0\Vert v\Vert^2_1,\quad v\in V.
\end{eqnarray}
Note that the expression of the nonlinear term $F$ has changed as we included the term $-c_0X$ in the new nonlinear term that we still denote by  $F$. The coercivity property (\ref{ellip2}) implies that $A$ is sectorial on $L^{2}(\Lambda)$, i.e.  there exist $C_{1},\, \theta \in (\frac{1}{2}\pi,\pi)$ such that
\begin{eqnarray}
 \Vert (\lambda I -A )^{-1} \Vert_{L(L^{2}(\Lambda))} \leq \dfrac{C_{1}}{\vert \lambda \vert },\;\quad \quad
\lambda \in S_{\theta},
\end{eqnarray}
where $S_{\theta}:=\left\lbrace  \lambda \in \mathbb{C} :  \lambda=\rho e^{i \phi},\; \rho>0,\;0\leq \vert \phi\vert \leq \theta \right\rbrace $ (see \cite{Henry}).
 Then  $A$ is the infinitesimal generator of a bounded analytic semigroup $S(t):=e^{t A}$  on $L^{2}(\Lambda)$  such that
\begin{eqnarray}
S(t):= e^{t A}=\dfrac{1}{2 \pi i}\int_{\mathcal{C}} e^{ t\lambda}(\lambda I - A)^{-1}d \lambda,\;\;\;\;\;\;\;
\;t>0,
\end{eqnarray}
where $\mathcal{C}$  denotes a path that surrounds the spectrum of $A $.
The coercivity  property \eqref{ellip2} also implies that $-A$ is a positive operator and its fractional powers are well defined  for any $\alpha>0,$ by
\begin{equation}
\label{fractional}
 \left\{\begin{array}{rcl}
         (-A)^{-\alpha} & =& \frac{1}{\Gamma(\alpha)}\displaystyle\int_0^\infty  t^{\alpha-1}{\rm e}^{tA}dt,\\
         (-A)^{\alpha} & = & ((-A)^{-\alpha})^{-1},
        \end{array}\right.
\end{equation}
where $\Gamma(\alpha)$ is the Gamma function (see \cite{Henry}).
Let's now turn to the discretization of our problem \eqref{model}.  We start by  splitting  the domain $\Lambda$ in finite triangles.
Let $\mathcal{T}_h$ be the triangulation with maximal length $h$ satisfying the usual regularity assumptions, and  $V_h \subset V$ the space of continuous functions that are 
piecewise linear over the triangulation $\mathcal{T}_h$. 
We consider the projection $P_h$ from $H=L^2(\Lambda)$ to $V_h$ defined for every $u\in H$ by 
\begin{eqnarray}
\label{projection}
\langle P_hu,\chi\rangle_H=\langle u,\chi\rangle_H, \quad \forall \chi\in V_h.
\end{eqnarray}
The discrete operator $A_h : V_h\longrightarrow V_h$ is defined by 
\begin{eqnarray}
\label{discreteoperator}
\langle A_h\phi,\chi\rangle_H=\langle A\phi,\chi\rangle_H=-a(\phi,\chi),\quad \forall \phi,\chi\in V_h,
\end{eqnarray}
Like $A$, $A_h$ is also a generator of a semigroup $S_h(t) : =e^{tA_h}$. 
As any semigroup and its generator, $A_h$ and $S_h(t)$ satisfy the smoothing properties of Proposition  \ref{theorem1}  with a uniform constant $C$, independent of $h$. 
Following \cite{Larsson1,Larsson2,Antonio2,Suzuki}, we characterize the domain of the operator $(-A)^{k/2},\, 1 \leq k\leq 2$ as follow 
\begin{eqnarray*}
\mathcal{D}((-A)^{k/2})=\mathbb{H}\cap H^{k}(\Lambda), \quad \text{ (for Dirichlet boundary conditions)},\\
\mathcal{D}(-A)=\mathbb{H}, \quad \mathcal{D}((-A)^{1/2})=H^1(\Lambda), \quad \text{(for Robin boundary conditions)}.
\end{eqnarray*}  
The semi-discrete  version of problem \eqref{model} consists to find $X^h(t)\in V_h$,  $t\in(0,T]$ such that $X^h(0)=P_hX_0$ and
\begin{eqnarray}
\label{semi1}
dX^h(t)=[A_hX^h(t)+P_hF(X^h(t))]dt+P_hB(X^h(t))dW(t), \quad t\in(0,T].
\end{eqnarray}
We note that $A_h$ and $P_hF$ satisfy the same assumptions as $A$ and $F$ respectively. We also note that $P_hB$ satisfies Assumption \ref{assumption4}. 
Therefore, Theorem \ref{theorem2} ensures the existence of the unique mild solution  $X^h(t)$ of \eqref{semi1} such that 
\begin{eqnarray}
\label{regular1}
\Vert X^h(t)\Vert\leq C(1+\Vert P_hX_0\Vert)\leq C(1+\Vert X_0\Vert), \quad \forall t\in[0,T].
\end{eqnarray}
The mild solution of \eqref{semi1} can be represented as follows 
\begin{eqnarray}
\label{mild2}
X^h(t)&=&S_h(t)X^h(0)+\int_0^tS_h(t-s)P_hF(X^h(s))ds\nonumber\\
&+&\int_0^tS_h(t-s)P_hB(X^h(s))dW(s).
\end{eqnarray}
The following lemma will be useful in our convergence analysis.
\begin{lemma}
\label{lemma1}
The following estimate holds 
\begin{eqnarray*}
\Vert (-A_h)^{\alpha}P_hv\Vert\leq C\Vert (-A)^{\alpha}v\Vert, \quad \forall\,\,\, 0\leq \alpha\leq \frac{1}{2}, \quad \forall\, v\in \mathcal{D}((-A)^{\alpha}).
\end{eqnarray*}
\end{lemma}
\begin{proof}
From the equivalence of norms (see \cite[(3.12)]{Larsson2}) we have 
\begin{eqnarray}
\label{inter1a}
\Vert (-A_h)^{1/2}P_h v\Vert\leq C\Vert P_hv\Vert_{H^1(\Lambda)}, \quad v\in H^1(\Lambda).
\end{eqnarray}
Note that
\begin{eqnarray}
\label{weak1}
\Vert P_h v\Vert^2_{H^1(\Lambda)}&=&\Vert P_h v\Vert^2_{L^2(\Lambda)}+\sum_{i=1}^d\left\Vert\frac{\partial (P_hv)}{\partial x_i}\right\Vert^2_{L^2(\Lambda)},
\end{eqnarray}
where $\frac{\partial}{\partial x_i}$ stands for the weak derivative. Let $\mathcal{D}(\Lambda)$ be the set of  functions $\varphi\in C^{\infty}(\Lambda)$ with compact support in $\Lambda$.
Let $ v \in L^2(\Lambda)$, for all $\varphi\in \mathcal{D}(\Lambda)$, we have
\begin{eqnarray}
\label{weak2}
\left\langle \frac{\partial (P_h v)}{\partial x_i},\varphi \right\rangle =-\left\langle P_h v, \frac{\partial \varphi}{\partial x_i}\right\rangle
=-\left\langle v, P_h^*\frac{\partial \varphi}{\partial x_i}\right\rangle,
\end{eqnarray}
where $\langle .,.\rangle$ is a duality pairing between $\mathcal{D}'(\Lambda)$ and $\mathcal{D}(\Lambda)$, and  $\dfrac{\partial \varphi}{\partial x_i}$ is the  derivative of $\varphi$ in the classical sense.
From \cite[Remark 2.1]{Antjd1} we have $$P_h^*\dfrac{\partial \varphi}{\partial x_i}=\dfrac{\partial (P_h^*\varphi)}{\partial x_i},$$ since $P_h^*$ is a linear operator.
So from  the equality  \eqref{weak2} it holds that
\begin{eqnarray}
\label{weak3}
\left\langle \frac{\partial (P_h v)}{\partial x_i},\varphi \right\rangle =-\left\langle v, \frac{\partial (P_h^*\varphi)}{\partial x_i}\right\rangle
=\left\langle \frac{\partial v}{\partial x_i}, P_h^*\varphi\right\rangle=\left\langle P_h\frac{\partial v}{\partial x_i}, \varphi\right\rangle.
\end{eqnarray}
Since \eqref{weak3} holds for all $\varphi\in \mathcal{D}(\Lambda)$, it follows that
\begin{eqnarray}
\dfrac{\partial (P_hv)}{\partial x_i}=P_h\dfrac{\partial v}{\partial x_i} \quad 
\text{in the weak sense}.
\end{eqnarray}
 Inserting this latter relation in \eqref{weak1}, using the fact that the projection $P_h$ 
is bounded with respect to the norm $\Vert .\Vert_{L^2(\Lambda)}$ and  again the equivalence of norm \cite[(3.12)]{Larsson2} yields 
\begin{eqnarray}
\label{weak4}
\Vert P_h v\Vert^2_{H^1(\Lambda)}&=&\Vert P_h v\Vert^2_{L^2(\Lambda)}+\sum_{i=1}^d\left\Vert P_h\frac{\partial v}{\partial x_i}\right\Vert^2_{L^2(\Lambda)}\nonumber\\
&=&\Vert v\Vert^2_{L^2(\Lambda)}+\sum_{i=1}^d\left\Vert \frac{\partial v}{\partial x_i}\right\Vert^2_{L^2(\Lambda)}\nonumber\\
&=&\Vert v\Vert^2_{H^1(\Lambda)}\leq C\Vert (-A)^{1/2}v\Vert.
\end{eqnarray}
 We therefore  have 
\begin{eqnarray}
\label{inter1}
\Vert (-A_h)^{1/2}P_h v\Vert\leq C\Vert (-A)^{1/2}v\Vert.
\end{eqnarray}
Note that \eqref{inter1} remains true if we replace $\frac{1}{2}$ by $0$. By interpolation theory we have
\begin{eqnarray}
\label{inter3}
\Vert (-A_h)^{\alpha}P_h v\Vert\leq C\Vert (-A)^{\alpha}v\Vert, \quad \forall \,\, 0\leq \alpha\leq \frac{1}{2}, \quad \forall v\in \mathcal{D}((-A)^{\alpha}).
\end{eqnarray}
\end{proof}
Let us recall the following well known lemma.
\begin{lemma}\textbf{[It\^{o} isometry]}\cite[Proposition 2.3.5]{Prevot}\\
For any  $t\in[0,T]$ and for any $L^0_2$-valued predictable process $\phi(s)$, $s\in[0,t]$ the following equality holds
\begin{eqnarray*}
\mathbb{E}\left[\left\Vert\int_0^t\phi(s)dW(s)\right\Vert^2\right]=
\mathbb{E}\left[\int_0^t\Vert\phi(s)
\Vert^2_{L^0_2}ds\right].
\end{eqnarray*}
\end{lemma}

The following two lemmas provide space and time regularity results of the mild solution of the semi-discrete problem  \eqref{semi1}. 
These lemmas play an important role in our convergence analysis. More results on the regularity of the mild solution of  problem \eqref{model} can be found in \cite{Arnulf1,Stig1,Prato}. 
\begin{lemma}\textbf{[Space regularity of the mild solution $X^h(t)$]}\\
\label{lemma3}
Let  \assref{assumption1}, \assref{assumption2}, \assref{assumption3} and \assref{assumption4} be fulfilled with $\beta\in[0,1)$, and $p\geq 2$.
Then for all $t \in [0,T]$, $X^h(t)\in L^p(\Omega, \mathcal{D}((-A)^{\beta/2}))$.  Moreover, there exists a positive constant $C$ independent of $h$ such that
\begin{eqnarray}
\label{bruit1}
\Vert (-A_h)^{\beta/2}X^h(t)\Vert_{L^p(\Omega, H)}\leq C\left(1+\Vert (-A)^{\beta/2}X_0\Vert_{L^p(\Omega, H)}\right),\quad t\in[0,T].
\end{eqnarray}
Further, when dealing with additive noise ($B=\mathbf{I}$), if \assref{assumption6a} is fulfilled with $\beta\in[0,2)$, then for all $t\in[0,T]$ $X^h(t)\in L^p(\Omega,\mathcal{D}((-A)^{\beta/2}))$ and \eqref{bruit1} holds for $\beta\in[0,2)$.
\end{lemma}

\begin{proof}
The proof follows the sames lines as that of \cite[Lemma 2.6]{Antonio1} or \cite[Theorem 1]{Arnulf1} or \cite[Theorem 3.1]{Stig1} by making use of \lemref{lemma1}.
\end{proof}

\begin{lemma}\textbf{[Time regularity of the mild solution $X^h(t)$]}
\label{lemma4}
Let $X^h$ be the mild solution of \eqref{semi1}. If \assref{assumption1}, \assref{assumption2}, \assref{assumption3} and \assref{assumption4} are fulfilled  with the corresponding $0< \beta \leq 2$.
For $ 0< \beta < 1,$ there exists a positive constant $C$ independent of $h$ such that for  $t_1,t_2\in[0,T]$, $t_1<t_2$, we have 
\begin{eqnarray}
\Vert X^h(t_2)-X^h(t_1)\Vert_{L^p(\Omega,H)}\leq C(t_2-t_1)^{\beta/2}(1+\Vert (-A)^{\beta/2}X_0\Vert_{L^p(\Omega,H)}).
\end{eqnarray}
Moreover, if Assumption \ref{assumption5} is fulfilled with $1\leq \beta \leq 2$, then there exists a positive constant $C$ such that 
\begin{eqnarray}
\Vert X^h(t_2)-X^h(t_1)\Vert_{L^p(\Omega,H)}\leq C(t_2-t_1)^{1/2}(1+\Vert(-A)^{\beta/2} X_0\Vert_{L^p(\Omega,H)}).
\end{eqnarray}
For additive noise ($B=\mathbf{I}$), if \assref{assumption1}, \assref{assumption2}, \assref{assumption3} and \assref{assumption6a} are fulfilled, then the following time regularity holds
\begin{eqnarray}
\Vert X^h(t_2)-X^h(t_1)\Vert_{L^p(\Omega,H)}\leq C(t_2-t_1)^{\min(\beta,1)/2}(1+\Vert(-A)^{\beta/2} X_0\Vert_{L^p(\Omega,H)}).\nonumber
\end{eqnarray}
\end{lemma}

\begin{proof}
The proof follows the same lines as that of \cite[Lemma 2.7]{Antonio1} or \cite[Theorem 1]{Arnulf1} or \cite[Theorem 4.1]{Stig1} by making use of Lemma \ref{lemma1}.
\end{proof}
\subsection{Current stable and efficient schemes for semilinear SPDEs}
\label{motivation}
Recall  that the simple  efficient standard semi-implicit Euler-Maruyama scheme  for (\ref{model}) is given by (see e.g. \cite{Printems})
\begin{eqnarray}
  \label{standard}
 Z_{m+1}^{h}&=& \Shdt 
\left[(\mathbf{I}+\Delta t(1-\theta)A_h) Z_{m}^{h} +\Delta t \,
P_{h}F(Z_{m}^{h})+P_{h} B(Z^h_m) \Delta W_{m}\right], \quad \\  \,\,\,\Shdt&=&\left(\mathbf{I}-\theta \Delta t  A_h \right)^{-1},\,\,\,\, \theta \in [0, 1].
\end{eqnarray}
The exponential integrators schemes developed in  \cite{Antonio1,ATthesis} are given by
\begin{eqnarray}
  \label{standard1}
 K_{m+1}^{h}&=& S_h(\Delta t) 
\left(K_{m}^{h} +\Delta t \,
P_{h}F(K_{m}^{h})+P_{h}B(K^h_m) \Delta W_{m} \right),\quad
\end{eqnarray}
and 
\begin{eqnarray}
  \label{standard2}
 L_{m+1}^{h}&=& S_h(\Delta t)
\left(L_{m}^{h}  \,
+P_{h} B(L^h_m) \Delta W_{m} \right)+ \Delta t \varphi_1( \Delta t A_h) P_{h}F(L_{m}^{h}),\\
\Delta W_{m}&:=& W_{t_{m+1}}-W_{t_{m}}=\sqrt{\Delta t} \underset{i \in \mathbb{N}}{\sum}\sqrt{q_{i}} R_{i,m}e_{i},\quad \nonumber
\end{eqnarray}
where  $R_{i,m}$ are independent, standard normally distributed random
variables with mean $0$ and variance $1$, and  the function $\varphi_1$ is defined in \eqref{phi}.  Note that all the initial values in all  the schemes  are taken  to be
$P_hX_0$ and scheme \eqref{standard} is more stable for $\theta  \geq 1/2$.
If  the linear operator $A$ tends to the operator null\footnote{Think about for example  a multiple of Laplace operator $A=\alpha \varDelta $,  when $\alpha \rightarrow 0$}, 
the corresponding discrete operator $A_h$ tends also to null, $S_h(\Delta t)$, $\Shdt$ and $\varphi_1( \Delta t A_h)$ tend to the identical operator $\mathbf{I}$. 
Therefore  the numerical schemes \eqref{standard}, \eqref{standard1} and \eqref{standard2} become the  unstable explicit   Euler-Maruyama scheme.

\subsection{Novel fully discrete scheme}
\label{fullydiscrete} 
Let us  build a more stable scheme, robust when the operator $A$ tends to null.
For the time discretization, we consider the one-step method which  provides the numerical approximated solution $X^h_m$ of $X^h(t_m)$ at discrete time $t_m=m\Delta t$, $m=0, \cdots, M$.
The method is based on the continuous linearization of \eqref{semi1}.
More precisely we linearize  \eqref{semi1}  at each time step as 
\begin{eqnarray}
\label{semi}
dX^h(t)&=&[A_hX^h(t)+J_m^hX^h(t)+G^h_m(X^h(t))]dt+P_hB(X^h(t))dW(t), 
\end{eqnarray}
for all $t_m\leq t\leq t_{m+1}$,
where $J_m^h$ is the Fr\'{e}chet derivative of $P_hF$ at $X^h_m$ and $G^h_m$ is  the remainder at $X^h_m$.
Both $J^h_m$ and $G^h_m$ are random functions and are  defined for all $\omega\in \Omega$ by
\begin{eqnarray}
\label{remainder1}
J^h_m(\omega) &:=&(P_hF)'(X^h_m(\omega))=P_hF'(X^h_m(\omega)),\\
\label{remainder2}
 G^h_m(\omega)(X^h(t)) &:=&P_hF(X^h(t))-J_m^h(\omega)X^h(t).
\end{eqnarray}
Before building  the new  numerical scheme, let us recall the following important lemma.
\begin{lemma}
\label{lemma5} 
 For all $m\in\mathbb{N}$ and all $\omega\in\Omega$, the random linear operator
$A_h+J^h_m(\omega)$ is the  generator of a strongly continuous semigroup $S^h_m(\omega)(t):=e^{(A_h+J^h_m(\omega))t}$ called  random (or stochastic) perturbed semigroup and uniformly bounded on $[0,T]$,
i.e. there exists a positive constant $C_1$ independent of $h$, $m$, $\Delta t$ and  the sample $\omega$ such that
\begin{eqnarray*}
\left\Vert e^{(A_h+J^h_m(\omega)) t}\right\Vert_{L(H)}\leq C_1, \quad  \;0\leq t\leq T.
\end{eqnarray*}
\end{lemma}  
\begin{proof}
  Using the boundedness of $P_h$ and \assref{assumption3}, it holds that
 \begin{eqnarray}
 \label{espoir1}
 \Vert J^h_m(\omega)\Vert_{L(H)} \leq   \Vert F'(X^h_m(\omega))\Vert_{L(H)} < C,\quad m\in\mathbb{N},\quad \omega\in\Omega.
 \end{eqnarray}
   Therefore $J^h_m(\omega)$ is a bounded linear operator. It follows then from  \cite[Theorem 1.1, Chapter 3, Page 76]{Pazy} that $A_h+J^h_m(\omega)$ is
   a generator of a strongly continuous semigroup denoted by $S^h_m(\omega)(t)=e^{(A_h+J^h_m(\omega))t}$.
   Since $A_h$ is a generator of an analytic semigroup  $S_h(t)=e^{A_ht}$, there exist two constants $K\geq 0$ and $C_0\in\mathbb{R}$ such that
 \begin{eqnarray}
 \label{espoir2}
 \Vert e^{A_h\,t}\Vert_{L(H)}\leq Ke^{C_0\,t}, \quad t\geq 0.
 \end{eqnarray}
    Finally using \eqref{espoir1} and \eqref{espoir2} it holds by applying again \cite[Theorem 1.1, Chapter 3, Page 76]{Pazy}) that
   \begin{eqnarray}
   \left\Vert e^{(A_h+J^h_m(\omega)) t}\right\Vert_{L(H)}&\leq& Ke^{\left(C_0+\Vert J^h_m(\omega)\Vert_{L(H)}\right)t}\nonumber\\
   &\leq& Ke^{(C_0+C)t}\leq C_1,\quad\quad\quad t\in[0,T],
   \end{eqnarray}
  where $C_1$ is a positive constant, independent of $h$, $m$, $\omega$ and $\Delta t$.  This complete the proof of \lemref{lemma5}.
\end{proof}

Given the solution $X^h(t_m)$ and the numerical solution $X^h_m$ at  $t_m$,  we obtain from \eqref{semi} the following mild representation form of $X^h(t_{m+1})$
\begin{eqnarray}
\label{semi2}
X^h(t_{m+1})&=&e^{(A_h+J^{h}_m)\Delta t}X^h(t_m)+\int_{t_m}^{t_{m +1}}e^{(A_h+J^{h}_m)(t_{m+1}-s)}G^h_{m}(X^h(s))ds\nonumber\\
&+&\int_{t_m}^{t_{m+1}}e^{(A_h+J^h_m)(t_{m+1}-s)}P_hB(X^h(s))dW(s).
\end{eqnarray}
We note that  \eqref{semi2} is the exact solution of \eqref{semi1} at $t_{m+1}$.
To establish our  numerical method we use the following approximations 
\begin{eqnarray}
\label{nna}
G^h_m(X^h(t_m+s))\approx G^h_m(X^h_m),\\
\label{napp}
 e^{(A_h+J^h_m)(t_{m+1}-s)}P_hB(X(s))\approx e^{(A_h+J^h_m)\Delta t}P_hB(X^h_m).
\end{eqnarray}
Therefore the deterministic  integral part of \eqref{semi2} can be approximated as follows 
\begin{eqnarray}
\label{constr1}
&&\int_{t_m}^{t_{m+1}}e^{(A_h+J^h_m)(t_{m+1}-s)}G^h_m(X^h(s))ds\nonumber\\
&=&\int_{0}^{\Delta t}e^{(A_h+J^h_m)(\Delta t-s)}G^h_m(X^h(t_m+s))ds\nonumber\\
&\approx& G^h_m(X^h_m)(A_h+J^h_m)^{-1}(e^{(A_h+J^h_m)\Delta t}-\mathbf{I}).
\end{eqnarray}
Inserting \eqref{constr1} and \eqref{napp} in \eqref{semi2} and using the approximation $X^h(t_m)\approx X^h_m$ give the following approximation $X^h_{m+1}$ of $X^h(t_{m+1})$, 
called Stochastic Exponential Rosenbrock Scheme (SERS)
\begin{eqnarray}
\label{erem}
X^h_{m+1}&=& e^{(A_h+J^h_m)\Delta t}X^h_m+(A_h+J^h_m)^{-1}(e^{(A_h+J^h_m)\Delta t}-\mathbf{I})G^h_m(X^h_m)\nonumber\\
&+&e^{(A_h+J^h_m)\Delta t}P_hB(X^h_m)(W_{t_{m+1}}-W_{t_m}),
\end{eqnarray}
with $X^h_0 : =X^h(0)=P_hX_0$. 
The  numerical scheme \eqref{erem} can be rewritten in the following equivalent form, which is efficient for implementation
\begin{eqnarray*}
X^h_{m+1}&=&X^h_m+P_hB(X^h_m)\Delta W_m\nonumber\\
&+&\varphi_1(\Delta t(A_h+J^h_m))\left[(A_h+J^h_m)(X^h_m+P_hB(X^h_m)\Delta W_m)+G^h_m(X^h_m)\right],
\end{eqnarray*}
where 
\begin{eqnarray}
\label{phi}
\varphi_1(\Delta t(A_h+J^h_m))&:=&(A_h+J^h_m)^{-1}(e^{\Delta t(A_h+J^h_m)}-\mathbf{I})\nonumber\\
&=&\int_0^{\Delta t}e^{(\Delta t-s)(A_h+J^h_m)}ds.
\end{eqnarray}
Note that the operator $\varphi_1(\Delta t(A_h+J^h_m(\omega)))$ is uniformly bounded (independently of $h$, $m$ and $\omega$), see e.g. \cite[Lemma 2.4]{Alex1}.

\begin{remark}
 Note that  the corresponding standard stochastic exponential  scheme  \eqref{standard2} presented  in \cite{Antonio1} can be written as
 \begin{eqnarray}
 \label{setd1}
L^h_{m+1}&=&L^h_m+P_hB(L^h_m)\Delta W_m\nonumber\\
&+&\varphi_1(\Delta tA_h)\left[A_h \left( L^h_m+P_hB(L^h_m)\Delta W_m\right)+P_hF(L^h_m)\right].
\end{eqnarray}
This scheme will be called SETD1 and will  be used in our numerical simulations for comparison with SERS scheme.
\end{remark}
\begin{remark}
 If  the deterministic part  is  also approximated as the diffusion part \eqref{napp}, we will obtain the following new scheme
 \begin{eqnarray}
 \label{schme0}
U^h_{m+1}=e^{(A_h+J^h_m)\Delta t} \left[ U^h_m+ P_hB(U^h_m)\Delta W_m+G^h_m(U^h_m))\right].
\end{eqnarray}
Our main result is also valid for scheme \eqref{schme0} and the extension of our proof  to that scheme is done as in \cite{Antonio1} without any issue.
\end{remark}

Having the numerical method  \eqref{erem}   in hand, our goal is to analyze its strong convergence toward the exact solution in the root-mean-square $L^2$ sense.
In the following subsection we state our strong convergence results, which are in fact our main results.

\subsection{Main results}
\label{strongresult}
Throughout this paper we take $t_m=m\Delta t\in[0,T]$, where $T=M\Delta t$ for $m, M\in\mathbb{N}$, $m\leq M$, $T$ is fixed, $C$ is a generic constant that may change from one place to another and $\epsilon>0$ is a positive constant small enough.
The main results of this paper are formulated in the following theorems.  For multiplicative noise  we have the following result.
\begin{theorem}
\label{mainresult1}
Let $X(t_m)$ and  $X^h_m$ be respectively the mild solution \eqref{mild1} and  the numerical approximation given by \eqref{erem} at $t_m=m\Delta t$. 
Let \assref{assumption1}, \assref{assumption2} (with $0<\beta< 1$ and $p=2$), \assref{assumption3} and \assref{assumption4} 
be fulfilled. Then the following error  estimate holds
\begin{eqnarray*}
(\mathbb{E}\Vert X(t_m)-X^h_m\Vert^2)^{1/2}\leq C\left(h^{\beta}+\Delta t^{\beta/2}\right).
\end{eqnarray*}
Moreover, under a strong  regularity of the initial data, that is  \assref{assumption2} (with $p=2$) and \assref{assumption5} are fulfilled with $\beta\in[1,2)$ and $\gamma=\beta-1$, the following error estimate holds
\begin{eqnarray*}
(\mathbb{E}\Vert X(t_m)-X^h_m\Vert^2)^{1/2}\leq C\left(h^{\beta}+\Delta t^{1/2}\right).
\end{eqnarray*}
If $\beta=2$ and \assref{assumption2} (with $p=2$) and \assref{assumption5} are fulfilled with $\gamma=1$, then the following error estimate holds
\begin{eqnarray*}
(\mathbb{E}\Vert X(t_m)-X^h_m\Vert^2)^{1/2}\leq C\left(h^{2}(1+\max(0,\ln(t_m/h^2)))+\Delta t^{1/2}\right).
\end{eqnarray*}
\end{theorem}

As in \cite[Remark 2.9]{Antonio1}, strong assumptions on the nonlinear function $F$ can allow to achieve  a spatial error of order $\mathcal{O}(h^2).$
Note that \assref{assumption5} in \thmref{mainresult1} is key to obtain optimal order of convergence. 
The following remark provides  the error estimate without  \assref{assumption5}.
\begin{remark}
\label{reviewremark1}
If \assref{assumption1}, \assref{assumption2} (with $p=2$), \assref{assumption3} and \assref{assumption4} are fulfilled with $\beta\in[1,2]$, then the following error estimate holds 
\begin{eqnarray*}
(\mathbb{E}\Vert X(t_m)-X^h_m\Vert^2)^{1/2}\leq C\left(h^{1-\epsilon}+\Delta t^{1/2-\epsilon}\right),
\end{eqnarray*}
where $\epsilon$ is a positive constant small enough.
\end{remark}

For additive noise (that is $B=\mathbf{I}$), we have the following result. 
\begin{theorem}
\label{mainresult2}
Let $X(t_m)$ and  $X^h_m$ be respectively the mild solution \eqref{mild1} and  the numerical approximation given by \eqref{erem} at $t_m=m\Delta t$. 
For additive noise, if  \assref{assumption1}, \assref{assumption2} (with $0<\beta< 2$ and $p=4$), \assref{assumption3},  \assref{assumption6a} and \assref{assumption6b} are fulfilled,
then the following error estimate holds
\begin{eqnarray}
(\mathbb{E}\Vert X(t_m)-X^h_m\Vert^2)^{1/2}\leq C(h^{\beta}+\Delta t^{\beta/2-\epsilon}).
\end{eqnarray}
 Moreover if  $\beta=2$, the following error estimate holds
\begin{eqnarray}
(\mathbb{E}\Vert X(t_m)-X^h_m\Vert^2)^{1/2}\leq C\left(h^{2}(1+\max(0,\ln(t_m/h^2)))+\Delta t^{1-\epsilon}\right).
\end{eqnarray}
\end{theorem}

\begin{remark}
\label{remark2a}
For additive noise, we achieved  suboptimal order $\mathcal{O}(h^{2-\epsilon}+\Delta t^{1-\epsilon})$ for $\beta=2$, 
where  $\epsilon$ is a positive number, small enough. The suboptimal order $1-\epsilon$ in time was  achieved in \cite{Jentzen2}, 
where authors imposed a strong regularity on the drift function (namely \cite[Assumption 2]{Jentzen2}), but with less regular noise.
The recent works in  \cite{Xiaojie1,Xiaojie2} achieved optimal order $1$ in time with less restrictive assumptions than \cite[Assumption 2]{Jentzen2}.
Here we have  achieved suboptimal order $1-\epsilon$ in time with similar assumptions  as in \cite{Xiaojie1,Xiaojie2}.
Note  that the current work and the work in \cite{Xiaojie2} use standard  Brownian increments, while the works in \cite{Jentzen2,Xiaojie1} 
use the linear functionals of the noise to achieve optimal order with less regular noise.
\end{remark}

\begin{remark}
\label{remark3}
Note that the semi-discrete problem \eqref{semi1} can be replaced by the following semi-discrete problem where  the noise is truncated 
\begin{eqnarray}
\label{semi3}
dX^h(t)=[A_hX^h(t)+P_hF(X^h(t))]dt+P_hB(X^h(t))P_hdW(t), t\in[0,T].
\end{eqnarray}
It was shown in \cite{Kovac1} that in the case of additive noise with smooth covariance operator kernel,
this truncation can be done severely without loosing the  spatial accuracy of the finite element method. 
Applying our stochastic exponential Rosenbrock scheme to \eqref{semi3} yields
\begin{eqnarray}
\label{serem2}
Y^h_{m+1}&=&e^{(A_h+J^h_m)\Delta t}Y^h_m+(A_h+J^h_m)^{-1}\left(e^{(A_h+J^h_m)\Delta t}-\mathbf{I}\right)G^h_m(Y^h_m)\nonumber\\
&+& e^{(A_h+J^h_m)\Delta t}P_hB(Y^h_m)P_h(W_{t_{m+1}}-W_{t_m}).
\end{eqnarray}
We note that \thmref{mainresult1} and \thmref{mainresult2} also hold for the numerical scheme \eqref{serem2}. Parts of  \cite{Antonio2} can be used in the proof.
\end{remark}
\section{Proof of the main results}
\label{convergenceproof}
Before  prove our main results,   some preparatory results are needed.
\subsection{Preparatory results}
\label{preliminaries1}

\begin{lemma}
\label{lemma6}
The function $G^h_m(\omega)$ defined by \eqref{remainder2}  satisfies the global Lipschitz condition with a uniform constant, i.e. 
there exists a positive constant $C>0$, independent of $h$, $m$ and $\omega$ such that
\begin{eqnarray*}
\Vert G^h_m(\omega)(u^h)-G^h_m(\omega)(v^h)\Vert\leq C\Vert u^h-v^h\Vert, \quad \forall m\in\mathbb{N}, \quad \forall u^h,v^h\in V_h.
\end{eqnarray*}
\end{lemma}
\begin{proof}
Using Assumption \ref{assumption3} and relations \eqref{remainder1}-\eqref{remainder2}, the proof is straightforward.
\end{proof}
We introduce the Riesz representation operator $R_h : V\longrightarrow V_h$ defined by
\begin{eqnarray}
\label{ritz}
\langle -AR_hv,\chi\rangle_H=\langle -Av,\chi\rangle_H=a(v,\chi),\quad \forall v\in V,\quad \forall \chi\in V_h.
\end{eqnarray}
It is well known (see \cite{Larsson2,Antonio1}) that $A$ and $A_h$ are related by   $A_hR_h=P_hA$. Under the 
regularity assumptions on the triangulation and in view of the $V$-ellipticity \eqref{ellip2}, it is well known (see \cite{Suzuki}) that for all $r\in\{1,2\}$ the following errors estimates hold 
\begin{eqnarray}
\label{ritz1}
\Vert R_hv-v\Vert+h\Vert R_hv-v\Vert_{H^1(\Omega)}\leq Ch^r\Vert v\Vert_{H^r(\Omega)},\quad v\in V\cap H^r(\Omega).
\end{eqnarray}

Let us consider the following deterministic linear problem :  find $u\in V$ such that 
\begin{eqnarray}
\dfrac{du}{dt}=Au,\quad u(0)=v,\quad t\in (0,T].
\end{eqnarray}
The corresponding semi-discrete problem in space consists  to find $u_h\in V_h$ such that 
\begin{eqnarray}
\dfrac{du_h}{dt}=A_hu_h,\quad u_h(0)=P_hv,\quad t\in (0,T].
\end{eqnarray}
Let us define the following operator 
\begin{eqnarray}
T_h(t):=S(t)-S_h(t)P_h=e^{At}-e^{A_ht}P_h,
\end{eqnarray}
so that $u(t)-u_h(t)=T_h(t)v$. The estimate \eqref{ritz1} was used in \cite{Antonio1} to prove  the key part of the following lemma.

\begin{lemma}  
\label{lemma6a}
The following estimate holds
\begin{eqnarray}
\label{addis1}
\Vert T_h(t)v\Vert \leq Ch^{r}t^{-(r-\alpha)/2}\Vert v\Vert_{\alpha},\quad r\in[0,2],\quad \alpha\leq r,\,\, t\in (0,T].
\end{eqnarray}
\end{lemma}
\begin{proof}
The proof of Lemma \ref{lemma6a} for $r\in [1,2]$ can be found in \cite[Lemma 3.1]{Antonio1}. Using the stability property of $S(t)$ and $S_h(t)$, 
and the fact that the projection $P_h$ is bounded, it follows that 
\begin{eqnarray}
\label{addis2}
\Vert S(t)v-S_h(t)P_hv\Vert\leq C\Vert v\Vert.
\end{eqnarray}
Inequality \eqref{addis2}  shows that \eqref{addis1}  holds for $r=0$. Interpolating between $r=0$ and $r=2$ completes the proof of Lemma \ref{lemma6a}.
\end{proof}

\begin{lemma} 
Let $X(t)$ and $X^h(t)$ be  the  mild solutions  given respectively by \eqref{mild1} and \eqref{mild2}. 
\label{lemma7}
\begin{itemize}
\item[(1)]
For multiplicative noise, assume that \assref{assumption1}, \assref{assumption2}, \assref{assumption3} and  \assref{assumption4}  are fulfilled. 
Then the following error estimate holds:
\begin{itemize}
\item[(i)] For  $0\leq\beta<1$
 \begin{eqnarray*}
\Vert X(t)-X^h(t)\Vert_{L^2(\Omega, H)}\leq Ch^{\beta}, \quad\, t \in (0,T].
\end{eqnarray*}
\item[(ii)] For $1\leq\beta<2$
\begin{eqnarray*}
\Vert X(t)-X^h(t)\Vert_{L^2(\Omega, H)}\leq Ch^{1-\epsilon}, \quad\, t \in (0,T],
\end{eqnarray*}
where $\epsilon$ is a positive constant small enough.
\item[(iii)]  For $1\leq\beta<2$, if moreover \assref{assumption5} is fulfilled with $\gamma=\beta-1$, we have 
\begin{eqnarray*}
\Vert X(t)-X^h(t)\Vert_{L^2(\Omega, H)}\leq Ch^{\beta}, \quad\, t \in (0,T].
\end{eqnarray*}
\item[(iv)] For  $\beta=2$ and if \assref{assumption5} is fulfilled with $\gamma=1$, we have 
\begin{eqnarray*}
\Vert X(t)-X^h(t)\Vert_{L^2(\Omega, H)}\leq Ch^2(1+\max(0,\ln(t/h^2))), \quad\, t \in (0,T].
\end{eqnarray*}
\end{itemize}
\item[(2)] For additive noise ($B=\mathbf{I}$), if \assref{assumption1}, \assref{assumption2}, \assref{assumption3} and \assref{assumption6a} are fulfilled,
then the following error estimate holds:
\begin{itemize}
\item[(i)] For $0\leq \beta<2$  
\begin{eqnarray}
\Vert X(t)-X^h(t)\Vert_{L^2(\Omega, H)}\leq Ch^{\beta}.
\end{eqnarray}
\item[(ii)] For $\beta=2$ 
\begin{eqnarray*}
\Vert X(t)-X^h(t)\Vert_{L^2(\Omega, H)}\leq Ch^2(1+\max(0,\ln(t/h^2))), \quad\, t \in (0,T].
\end{eqnarray*}
\end{itemize}
\end{itemize}
\end{lemma}

\begin{proof}
The proof of (1) (i) and (iii) can be found in \cite[Theorem 6.1]{Antonio2}. The proof of (1) (ii) is similar to that of \cite[Theorem 6.1]{Antonio2} using \lemref{lemma6a}.
The proof of (2) (i) and (ii) can be found in \cite[Proposition 3.3]{Kovac1}.
\end{proof}

\begin{lemma}
\label{lemma8}
Under  \assref{assumption1}, for all $\omega\in\Omega$,  the stochastic perturbed semigroup $S^h_m(\omega)(t)$ satisfies the following stability properties
\begin{itemize}
\item[(i)] For  $ \gamma_1, \gamma_2\leq 1, \text{such that}\,\,\, 0\leq \gamma_1+\gamma_2\leq 1$, we have
\begin{eqnarray*}
\Vert (-A_h)^{-\gamma_1}(S^h_m(\omega)(t)-\mathbf{I})(-A_h)^{-\gamma_2}\Vert_{L(H)}\leq Ct^{\gamma_1+\gamma_2},\,\, t \in (0,T]. 
\end{eqnarray*}
\item[(ii)] For $\gamma_1\geq 0$, we have  \begin{eqnarray*} 
\Vert S^h_m(\omega)(t)(-A_h)^{\gamma_1}\Vert_{L(H)}\leq Ct^{-\gamma_1},\,\, t \in (0,T],\,\, \gamma_1\geq 0,
\end{eqnarray*}
\item[(iii)] For  $\gamma_1\geq 0$ and $0\leq \gamma_2<1$ such that $\gamma_2-\gamma_1\geq 0$, we have
\begin{eqnarray*}
\Vert (-A_h)^{-\gamma_1}S^h_m(\omega)(t)(-A_h)^{\gamma_2}\Vert_{L(H)}\leq Ct^{\gamma_1-\gamma_2},\,\, t \in (0,T]. 
\end{eqnarray*}
\item[(iv)] For $\gamma_1, \gamma_2>0$  such that $0\leq \gamma_2-\gamma_1\leq 1$, then the following estimate holds 
\begin{eqnarray*}
\Vert (-A_h)^{-\gamma_1}(S^h_m(\omega)(t)-\mathbf{I})(-A_h)^{\gamma_2}\Vert_{L(H)}\leq Ct^{\gamma_1-\gamma_2},\,\, t \in (0,T]. 
\end{eqnarray*}
\end{itemize}
where $C$ is a positive constant independent of $h$, $m$, $\Delta t$ and the sample  $\omega$.
\end{lemma}
\begin{proof}
We recall that the perturbed semigroup satisfies the following variation of parameters 
formula (see \cite[Chapter 3, Corollary 1.7]{Klaus} or \cite[Section 3.1, Page 77]{Pazy})
\begin{eqnarray}
\label{vp1a}
S^h_m(\omega)(t)v=S_h(t)v+\int_0^tS_h(t-s)J^h_m(\omega)S^h_m(\omega)(s)vds,
\end{eqnarray}
for all $v\in H$ and all $t\geq 0$. Then it follows from \eqref{vp1a} that 
\begin{eqnarray}
\label{vp1}
(S^h_m(\omega)(t)-\mathbf{I})v=(S_h(t)-\mathbf{I})v+\int_0^tS_h(t-s)J^h_m(\omega)S^h_m(\omega)(s)vds.
\end{eqnarray}
It is obvious  that  $(-A_h)^{-\gamma_2}v\in H$ for all $v\in H$. Then, replacing $v$ in \eqref{vp1} by $(-A_h)^{-\gamma_2}v$
and pre-multiplying both right-hand sides of \eqref{vp1} by $(-A_h)^{-\gamma_1}$ yields 
\begin{eqnarray}
\label{vp2}
&&(-A_h)^{-\gamma_1}(S^h_m(\omega)(t)-\mathbf{I})(-A_h)^{-\gamma_2}v\nonumber\\
&=&(S_h(t)-\mathbf{I})(-A_h)^{-\gamma_2-\gamma_1}v\\
&+&\int_0^t(-A_h)^{-\gamma_1}S_h(t-s)J^h_m(\omega)S^h_m(\omega)(s)(-A_h)^{-\gamma_2} v ds\nonumber.
\end{eqnarray}
Taking the norm in both sides of \eqref{vp2} and using  \propref{prop1}, the fact that  $(-A_h)^{-\gamma_2}$ and $J^h_m(\omega)$ are uniformly bounded, it follows that 
\begin{eqnarray*}
\Vert (-A_h)^{-\gamma_1}(S^h_m(\omega)(t)-\mathbf{I})(-A_h)^{-\gamma_2}v\Vert &\leq& Ct^{\gamma_2+\gamma_1}\Vert v\Vert+C\int_0^t\Vert v\Vert ds\nonumber\\
&\leq& Ct^{\gamma_2+\gamma_1}\Vert v\Vert.
\end{eqnarray*}
Using the definition of the norm $\Vert . \Vert_{L(H)}$ gives the desired result for (i). To prove (ii), we replace $v$ by $(-A)^{\gamma_1}v$ in \eqref{vp1a} and obtain
\begin{eqnarray}
\label{vp3a}
S^h_m(\omega)(t)(-A_h)^{\gamma_1}v&=&S_h(t)(-A_h)^{\gamma_1}v\nonumber\\
&+&\int_0^tS_h(t-s)J^h_m(\omega)S^h_m(\omega)(s)(-A_h)^{\gamma_1}vds, 
\end{eqnarray}
for all $v\in H$ and all $t\geq 0$. Taking the norm in both sides of \eqref{vp3a} and using the stability property of $S_h(t)$, $S^h_m(\omega)(t)$ 
with the uniformly  boundedness of $J^h_m(\omega)$ gives 
\begin{eqnarray}
\label{vp3b}
\Vert S^h_m(\omega)(t)(-A_h)^{\gamma_1}v\Vert&\leq& Ct^{-\gamma_1}\Vert v\Vert\nonumber\\
&+&C\int_0^t\Vert S^h_m(\omega)(s)(-A_h)^{\gamma_1}\Vert_{L(H)}\Vert v\Vert ds.
\end{eqnarray}
From \eqref{vp3b} it holds that
\begin{eqnarray}
\label{vp3c}
\Vert S^h_m(\omega)(t)(-A_h)^{\gamma_1}\Vert_{L(H)}\leq Ct^{-\gamma_1}+C\int_0^t\Vert S^h_m(\omega)(s)(-A_h)^{\gamma_1}\Vert_{L(H)} ds.
\end{eqnarray}
Applying the continuous Gronwall's lemma to \eqref{vp3c} completes the proof of (ii).  To prove (iii), we replace $v$ in \eqref{vp1a} by $(-A_h)^{\gamma_2}v$ and   pre-multiply both sides by $(-A_h)^{-\gamma_1}$. This yields
\begin{eqnarray}
\label{espoir3}
(-A_h)^{-\gamma_1}S^h_m(\omega)(t)(-A_h)^{\gamma_2}v&=&(-A_h)^{-\gamma_1}S_h(t)(-A_h)^{\gamma_2}v\\
&+&\int_0^t(-A_h)^{-\gamma_1}S_h(t-s)J^h_m(\omega)S^h_m(\omega)(s)(-A_h)^{\gamma_2}vds.\nonumber
\end{eqnarray}
Taking the norm in both sides of \eqref{espoir3}, using the stability properties of \propref{prop1}, the boundedness of $(-A_h)^{-\gamma_1}$, $J^h_m$ and applying \lemref{lemma8} (ii), it holds that
\begin{eqnarray}
\Vert (-A_h)^{-\gamma_1}S^h_m(\omega)(t)(-A_h)^{\gamma_2}v\Vert&\leq& \Vert (-A_h)^{-\gamma_1}S_h(t)(-A_h)^{\gamma_2}v\Vert\nonumber\\
&+&C\int_0^t\Vert S^h_m(\omega)(s)(-A_h)^{\gamma_2}v\Vert ds\nonumber\\
&\leq& Ct^{\gamma_1-\gamma_2}\Vert v\Vert+C\int_0^ts^{-\gamma_2}ds\Vert v\Vert\nonumber\\
&\leq& C(t^{\gamma_1-\gamma_2}+t^{1-\gamma_2})\Vert v\Vert,
\end{eqnarray}
and for $t\leq T$, since $\gamma_1\leq\gamma_2\leq 1$, $t^{1-\gamma_2}\leq C(T)t^{\gamma_1-\gamma_2}$. 
This ends the proof of (iii). The proof of (iv) is similar to that of (i).
\end{proof}

The following lemma is similar to \cite[Lemma 4]{Julia1}, but its proof is easier than that of \cite[Lemma 4]{Julia1} since we do not use any further lemmas in its proof.
\begin{lemma}
\label{lemma9}
Under  \assref{assumption1} and \assref{assumption3}, the perturbed semigroup $S^h_m(\omega)$ satisfies the following stability property
\begin{eqnarray*}
\left\Vert e^{(A_h+J^h_m(\omega))\Delta t}\cdots e^{(A_h+J^h_k(\omega))\Delta t}(-A_h)^{\nu}\right\Vert_{L(H)}\leq Ct_{m+1-k}^{-\nu}, \quad 0\leq \nu< 1,
\end{eqnarray*}
where $C$ is a positive constant independent of $m$, $k$, $h$, $\Delta t$ and  the sample $\omega$.
\end{lemma}
\begin{proof}
As in \cite{Antjd1} we set 
\begin{eqnarray*}
\left\{\begin{array}{ll}
S^h_{m,k}(\omega) := e^{(A_h+J^h_m(\omega))\Delta t}\cdots e^{(A_h+J^h_k(\omega))\Delta t}, \quad \text{if} \quad m\geq k\\
S^h_{m,k}(\omega) :=\mathbf{I}, \hspace{5cm} \text{if}\quad m<k
\end{array}
\right.
\end{eqnarray*}
Using the telescoping sum, we can rewrite the composition of the perturbed semigroup  $S^h_{m,k}(\omega)$ as follow
\begin{eqnarray}
\label{pertur1}
S^h_{m,k}(\omega)&=&e^{A_h(t_{m+1-k})}+e^{A_h(t_{m+1}-t_{k+1})}\left(e^{(A_h+J^h_k(\omega))\Delta t}-e^{A_h\Delta t}\right)\nonumber\\
&+&\sum_{j=k+1}^{m}e^{A_h(t_{m+1}-t_{j+1})}\left(e^{(A_h+J^h_j(\omega))\Delta t}-e^{A_h\Delta t}\right)S^h_{j-1,k}(\omega).
\end{eqnarray}
Multiplying both sides of \eqref{pertur1} by  $(-A_h)^{\nu}$ yields 
\begin{eqnarray}
\label{pertur2}
&&S^h_{m,k}(\omega)(-A_h)^{\nu}\nonumber\\
&=&e^{A_ht_{m+1-k}}(-A_h)^{\nu}+e^{A_h(t_{m+1}-t_{k+1})}\left(e^{(A_h+J^h_k(\omega))\Delta t}-e^{A_h\Delta t}\right)(-A_h)^{\nu}\nonumber\\
&+&\sum_{j=k+1}^{m}e^{A_h(t_{m+1}-t_{j+1})}\left(e^{(A_h+J^h_j(\omega))\Delta t}-e^{A_h\Delta t}\right)S^h_{j-1,k}(\omega)(-A_h)^{\nu}.
\end{eqnarray}
Using the variation of parameter formula \eqref{vp1a}, the fact that the Jacobian, the semigroup $S^h_m(\omega)(t)$ and $S_h(t)$ are  uniformly bounded, we obtain 
\begin{eqnarray}
\label{atrejd}
\left\Vert e^{(A_h+J^h_m(\omega))\Delta t}-e^{A_h\Delta t}\right\Vert_{L(H)}\leq C\Delta t
\end{eqnarray}
Taking the norm in both sides of \eqref{pertur2} and using the stability property of $S_h(t)$ 
together with \eqref{atrejd}  gives
\begin{eqnarray}
\label{pertur3}
&&\left \Vert S^h_{m,k}(\omega)(-A_h)^{\nu}\right\Vert_{L(H)}\nonumber\\
&\leq& Ct^{-\nu}_{m+1-k}+\left\Vert e^{A_h(t_{m+1}-t_{k+1})}\right\Vert_{L(H)}\left\Vert\left(e^{(A_h+J^h_k(\omega))\Delta t}-e^{A_h\Delta t}\right)(-A_h)^{\nu}\right\Vert_{L(H)}\nonumber\\ &+&\sum_{j=k+1}^{m}\Vert e^{A_h(t_{m+1}-t_{j+1})}\Vert_{L(H)}\left\Vert e^{(A_h+J^h_m(\omega))\Delta t}-e^{A_h\Delta t}\right\Vert_{L(H)}\Vert S^h_{j-1,k}(\omega)(-A_h)^{\nu}\Vert_{L(H)}\nonumber\\
&\leq&Ct^{-\nu}_{m+1-k}+C\left\Vert\left(e^{(A_h+J^h_k(\omega))\Delta t}-e^{A_h\Delta t}\right)(-A_h)^{\nu}\right\Vert_{L(H)}\nonumber\\
&+&C\Delta t\sum_{j=k+1}^{m}\Vert S^h_{j-1,k}(\omega)(-A_h)^{\nu}\Vert_{L(H)}.
\end{eqnarray}
Rewriting \eqref{vp1a} with $t=\Delta t$ yields
\begin{eqnarray}
\label{pertur4}
e^{(A_h+J^h_m(\omega))\Delta t}-e^{A_h\Delta t}=\int_0^{\Delta t}e^{A_h(\Delta t -s)}J^h_m(\omega)e^{(A_h+J^h_m(\omega))s}ds.
\end{eqnarray}
Multiplying both sides of \eqref{pertur4} by $(-A_h)^{\nu}$ gives
\begin{eqnarray}
\label{pertur5}
&&\left(e^{(A_h+J^h_m(\omega))\Delta t}-e^{A_h\Delta t}\right)(-A_h)^{\nu}\nonumber\\
&=&\int_0^{\Delta t}e^{A_h(\Delta t -s)}J^h_m(\omega)e^{(A_h+J^h_m(\omega))s}(-A_h)^{\nu}ds.
\end{eqnarray}
Taking the norm in both sides of \eqref{pertur5}, using the stability property of $e^{A_ht}$, the uniform boundedness of $J^h_m$   and Lemma \ref{lemma8} (ii) with $\gamma_1=\nu$  gives
\begin{eqnarray}
\label{pertur6}
&&\left\Vert\left(e^{(A_h+J^h_m(\omega))\Delta t}-e^{A_h\Delta t}\right)(-A_h)^{\nu}\right\Vert_{L(H)}\nonumber\\
&\leq&\int_0^{\Delta t}\Vert e^{A_h(\Delta t -s)}\Vert_{L(H)}\Vert J^h_m(\omega)\Vert_{L(H)}\Vert e^{(A_h+J^h_m(\omega))s}(-A_h)^{\nu}\Vert_{L(H)}ds\nonumber\\
&\leq& C\int_0^{\Delta t}s^{-\nu}ds\leq C\Delta t^{1-\nu}=Ct_1^{-\nu}\Delta t.
\end{eqnarray}
Substituting \eqref{pertur6} in \eqref{pertur3} yields 
\begin{eqnarray}
\label{pertur7}
\Vert S^h_{m,k}(\omega)(-A_h)^{\nu}\Vert_{L(H)}
&\leq& Ct^{-\nu}_{m+1-k}+Ct_1^{-\nu}\Delta t\Vert \mathbf{I}\Vert_{L(H)}\nonumber\\
&+&C\Delta t\sum_{j=k+1}^{m}\Vert S^h_{j-1,k}(\omega)(-A_h)^{\nu}\Vert_{L(H)}
\end{eqnarray}
Applying the discrete Gronwall's inequality to \eqref{pertur7} completes the proof of \lemref{lemma9}.
\end{proof}

\begin{lemma}
\label{lemma10}
If \assref{assumption6a} is  fulfilled, then the following estimate holds
\begin{eqnarray}
\label{addiv1}
\left\Vert (-A_h)^{\frac{\beta-1}{2}}P_hQ^{\frac{1}{2}}\right\Vert^2_{\mathcal{L}_2(H)}<C, 
\end{eqnarray}
where $\beta$ is defined in \assref{assumption2}.
\end{lemma}
\begin{proof}
The proof when $0\leq\beta\leq 1$ can be found in \cite[Proposition 4.1]{Antonio2}. To prove  \eqref{addiv1} when $1<\beta\leq 2$, we use the definition of $\Vert .\Vert_{\mathcal{L}_2(H)}$,  apply \lemref{lemma1} with $\alpha=\frac{\beta-1}{2}$  and  \assref{assumption6a} to get
\begin{eqnarray}
\left\Vert (-A_h)^{\frac{\beta-1}{2}}P_hQ^{\frac{1}{2}}\right\Vert^2_{\mathcal{L}_2(H)}&=&\sum_{i=1}^{\infty}\left\Vert (-A_h)^{\frac{\beta-1}{2}}P_hQ^{\frac{1}{2}}e_i\right\Vert^2\nonumber\\
&\leq& C\sum_{i=1}^{\infty}\left\Vert (-A)^{\frac{\beta-1}{2}}Q^{\frac{1}{2}}e_i\right\Vert^2\nonumber\\
&=&C\left\Vert (-A)^{\frac{\beta-1}{2}}Q^{\frac{1}{2}}\right\Vert^2_{\mathcal{L}_2(H)}\leq C.
\end{eqnarray}
\end{proof}
\begin{lemma}
\label{lemma11}
Under \assref{assumption3} and \assref{assumption6b}, for all $\omega\in\Omega$, the following estimates hold
\begin{eqnarray}
\label{tamb1}
\Vert (G^h_k(\omega))'(u)v\Vert\leq C\Vert v\Vert, \quad u,v\in H,\quad k\in\mathbb{N},\\
\label{tamb2}
\Vert (-A_h)^{\frac{-\eta}{2}}(G^h_k(\omega))''(u)(v_1,v_2)\Vert\leq C\Vert v_1\Vert.\Vert v_2\Vert,\quad v_1,v_2\in H,\quad k\in\mathbb{N},
\end{eqnarray}
for some  $\eta \in [1,2)$.
\end{lemma}
\begin{proof}
The proof follows the same lines as \cite[Proposition 4.1]{Antonio2}. Indeed since the Jacobian  $J^h_k(\omega)$ is a linear operator, 
taking the differential in both sides of  \eqref{remainder2} yields
\begin{eqnarray}
\label{tamb4}
(G^h_k(\omega))'(u)=P_hF'(u)-J^h_k(\omega)=P_hF'(u)-P_hF'(X^h_k(\omega)),
\end{eqnarray}
and therefore 
\begin{eqnarray}
\label{tamb5}
(G^h_k(\omega))'(u)v=P_hF'(u)v-P_hF'(X^h_k(\omega))v,\quad v\in H.
\end{eqnarray}
Taking the norm in both sides of \eqref{tamb5} and using \assref{assumption3} yields 
\begin{eqnarray}
\label{tamb6}
\Vert (G^h_k(\omega))'(u)v\Vert\leq \Vert P_hF'(u)v\Vert+\Vert P_hF'(X^h_k(\omega))v\Vert \leq C\Vert v\Vert,
\end{eqnarray}
which proves \eqref{tamb1}. Taking the differential at the point $u\in H$  in both sides of \eqref{tamb4} yields 
\begin{eqnarray}
\label{tamb7}
(G^h_k(\omega))''(v_1,v_2)=P_hF''(u)(v_1,v_2), \quad v_1,v_2\in H.
\end{eqnarray}
Taking the norm in both sides of \eqref{tamb7}, using \cite[(70)]{Antonio2} and \assref{assumption6b} yields
\begin{eqnarray}
\Vert (-A_h)^{\frac{-\eta}{2}}(G^h_k(\omega))''(u)(v_1,v_2)\Vert&=&\Vert(-A_h)^{\frac{-\eta}{2}} P_hF''(u)(v_1,v_2)\Vert\nonumber\\
&\leq& \Vert (-A)^{\frac{-\eta}{2}}F''(u)(v_1,v_2)\Vert\nonumber\\
&\leq& C\Vert v_1\Vert.\Vert v_2\Vert.
\end{eqnarray}
\end{proof}

Gathering our preparatory results,  we are now ready to prove our main result in \thmref{mainresult1}.
\subsection{Proof of \thmref{mainresult1}}
\label{proofmainresult1}
Using the standard  technique in the error analysis, we  split the fully discrete error in two terms
\begin{eqnarray*}
\Vert X(t_m)-X^h_m\Vert_{L^2(\Omega, H)}&\leq& \Vert X(t_m)-X^h(t_m)\Vert_{L^2(\Omega, H)}+\Vert X^h(t_m)-X^h_m\Vert_{L^2(\Omega, H)}\nonumber\\
&=:&err_0+err_1.
\end{eqnarray*}
Note that  the space  error $err_0$ is   estimated by Lemma \ref{lemma7}. It remains to estimate the time error $err_1$. 
We estimate the time error $err_1$ for both $0\leq\beta<1$ and $1\leq \beta<2$ separately in the following two subsections.
\subsubsection{Estimate of the time error for $0\leq \beta<1$}
\label{firstestimate}
 We recall that the exact solution at $t_m$ of the semidiscrete  problem is given by 
\begin{eqnarray}
\label{num1}
X^h(t_m)&=&e^{(A_h+J_{m-1}^h)\Delta t}X^h(t_{m-1})\nonumber\\
&+&\int_{t_{m-1}}^{t_m}e^{(A_h+J_{m-1}^h)(t_m-s)}G^h_{m-1}(X^h(s))ds\nonumber\\
&+&\int_{t_{m-1}}^{t_m}e^{(A_h+J_{m-1}^h)(t_m-s)}P_hB(X^h(s))dW(s).
\end{eqnarray}
We also recall that the numerical solution at $t_m$  given by \eqref{erem} can be rewritten as 
\begin{eqnarray}
\label{num2}
X^h_m&=&e^{(A_h+J_{m-1}^h)\Delta t}X^h_{m-1}\nonumber\\
&+&\int_{t_{m-1}}^{t_m}e^{(A_h+J_{m-1}^h)(t_m-s)}G^h_{m-1}(X^h_{m-1})ds\nonumber\\
&+&\int_{t_{m-1}}^{t_m}e^{(A_h+J_{m-1}^h)\Delta t}P_hB(X^h_{m-1})dW(s).
\end{eqnarray}
If $m=1$ then it follows from \eqref{num1} and \eqref{num2} that 
\begin{eqnarray}
\label{addition1}
&&\Vert X(t_1)-X^h_1\Vert_{L^2(\Omega, H)}\nonumber\\
&\leq&\left\Vert \int_0^{\Delta t}e^{(A_h+J^h_0)(\Delta t-s)}[G^h_0(X^h(s))-G^h_0(X^h_0)]ds\right\Vert_{L^2(\Omega, H)}\nonumber\\
&+&\left\Vert\int_0^{\Delta t}\left[e^{(A_h+J^h_0)(\Delta t-s)}P_hB(X^h(s))-e^{(A_h+J^h_0)\Delta t}P_hB(X^h_0)\right]dW(s)\right\Vert_{L^2(\Omega, H)}\nonumber\\
&=:& I+II.
\end{eqnarray}
Using Lemma \ref{lemma5}, Lemma \ref{lemma6}, \eqref{regular1} and the fact that  $X^h_0=P_hX_0$ we obtain the following estimate
\begin{eqnarray}
\label{addition3}
I\leq C\Delta t.
\end{eqnarray}
Using the It\^{o}'s isometry property, triangle inequality, Lemma \ref{lemma5}, Assumption \ref{assumption4}, \eqref{regular1} and 
the fact that  $(a+b)^2\leq 2a^2+2b^2$ for all $a,b\in\mathbb{R}$, $\sqrt{u+v}\leq \sqrt{u}+\sqrt{v}$ for all positive real numbers $u$ and $v$, we obtain the following estimate
\begin{eqnarray}
\label{addition5}
II
&\leq& C\Delta t^{1/2}.
\end{eqnarray}
Inserting \eqref{addition5} and \eqref{addition3} in \eqref{addition1} yields 
\begin{eqnarray}
\label{addition6}
\Vert X^h(t_1)-X^h_1\Vert_{L^2(\Omega, H)}\leq C\Delta t^{1/2}.
\end{eqnarray}
For $m\geq 2$, we iterate the mild solution  \eqref{num1} at $t_m$ by substituting $X^h(t_j),\,j =1,2,..,m-1$  in  \eqref{num1} by their  mild forms. We obtain
{ \small {
\begin{eqnarray}
\label{num1a}
\lefteqn {X^h(t_m)}&& \nonumber\\
&=&e^{(A_h+J_{m-1}^h)\Delta t}\cdots e^{(A_h+J^h_0)\Delta t} X^h(0)+\int_{t_{m-1}}^{t_m}e^{(A_h+J^h_{m-1})(t_m-s)}G^h_{m-1}(X^h(s))ds \nonumber\\
&+&\int_{t_{m-1}}^{t_m}e^{(A_h+J_{m-1}^h)(t_m-s)}P_hB(X^h(s))dW(s)\\
&+&\sum_{k=0}^{m-2}\int_{t_{m-k-2}}^{t_{m-k-1}} e^{(A_h+J^h_{m-1})\Delta t}\cdots e^{(A_h+J^h_{m-k-1})\Delta t}e^{(A_h+J^h_{m-k-2})(t_{m-k-1}-s)}G^h_{m-k-2}(X^h(s))ds\nonumber\\
&+&\sum_{k=0}^{m-2}\int_{t_{m-k-2}}^{t_{m-k-1}}e^{(A_h+J^h_{m-1})\Delta t}\cdots e^{(A_h+J^h_{m-k-1})\Delta t}e^{(A_h+J^h_{m-k-2})(t_{m-k-1}-s)}P_hB(X^h(s))dW(s) \nonumber.
\end{eqnarray}
}}
Similarly,  for $m\geq 2$, we iterate the numerical solution \eqref{num2} at $t_m$ by substituting $X^h_j,\,j =1,2,..,m-1$  only in the first term of \eqref{num2} by their expressions. We obtain
{ \small {
\begin{eqnarray}
\label{num2a}
\lefteqn{X^h_m}\nonumber\\
&=&e^{(A_h+J_{m-1}^h)\Delta t}\cdots e^{(A_h+J^h_0)\Delta t} X^h(0)+\int_{t_{m-1}}^{t_m}e^{(A_h+J^h_{m-1})(t_m-s)}G^h_{m-1}(X^h_{m-1})ds\nonumber\\
&+&\int_{t_{m-1}}^{t_m}e^{(A_h+J_{m-1}^h)\Delta t}P_hB(X^h_{m-1})dW(s)\nonumber\\
&+&\sum_{k=0}^{m-2}\int_{t_{m-k-2}}^{t_{m-k-1}}e^{(A_h+J^h_{m-1})\Delta t}\cdots e^{(A_h+J^h_{m-k-1})\Delta t}e^{(A_h+J^h_{m-k-2})(t_{m-k-1}-s)} \\
  &&G^h_{m-k-2}(X^h_{m-k-2})ds\nonumber\\
&+&\sum_{k=0}^{m-2}\int_{t_{m-k-2}}^{t_{m-k-1}}e^{(A_h+J^h_{m-1})\Delta t}\cdots e^{(A_h+J^h_{m-k-1})\Delta t}e^{(A_h+J^h_{m-k-2})\Delta t}
 P_hB(X^h_{m-k-2})dW(s).\nonumber
\end{eqnarray}
}}
Therefore, it follows from \eqref{num1a} and \eqref{num2a} and  the triangle inequality that 
\begin{eqnarray}
\label{somme1}
\dfrac{1}{4}\Vert X^h(t_m)-X^h_m\Vert^2_{L^2(\Omega, H)} \leq III+IV+V+VI,
\end{eqnarray}
where
{ \small {
\begin{eqnarray*}
III&=& \left\Vert\int_{t_{m-1}}^{t_m}e^{(A_h+J^h_{m-1})(t_m-s)}\left[G^h_{m-1}(X^h(s))-G^h_{m-1}(X^h_{m-1})\right]ds\right\Vert^2_{L^2(\Omega, H)},\\
IV&=&\left\Vert\int_{t_{m-1}}^{t_m}\left(e^{(A_h+J^h_{m-1})(t_m-s)}P_hB(X^h(s))-e^{(A_h+J^h_{m-1})\Delta t}P_hB(X^h_{m-1})\right)dW(s)\right\Vert^2_{L^2(\Omega, H)},\\
\end{eqnarray*}
}}
{ \small {
\begin{eqnarray*}
V&=&\left\Vert\sum_{k=0}^{m-2}\int_{t_{m-k-2}}^{t_{m-k-1}}e^{(A_h+J^h_{m-1})\Delta t}\cdots e^{(A_h+J^h_{m-k-1})\Delta t}e^{(A_h+J^h_{m-k-2})(t_{m-k-1}-s)}\right. \\
&&\left.(G^h_{m-k-2}(X^h(s))-G^h_{m-k-2}(X^h_{m-k-2}))ds\right \Vert^2_{L^2(\Omega, H)},\\
VI&=&\left\Vert\sum_{k=0}^{m-2}\int_{t_{m-k-2}}^{t_{n-k-1}}e^{(A_h+J^h_{m-1})\Delta t}\cdots e^{(A_h+J^h_{m-k-1})\Delta t}\right. \\
&&\left.\left(e^{(A_h+J^h_{m-k-2})(t_{m-k-1}-s)}P_hB(X^h(s))-e^{(A_h+J^h_{m-k-1})\Delta t}P_hB(X^h_{m-k-2})\right)dW(s)\right\Vert^2_{L^2(\Omega, H)}.
\end{eqnarray*}
}}
Using Holder's inequality, the stability property of $S^h_m(t)$, \lemref{lemma6}, the  triangle inequality and the fact  that  $(a+b)^2\leq 2a^2+2b^2$ yields
{ \small {
\begin{eqnarray}
\label{num2aaa}
&&III\nonumber\\
&\leq&\left( \int_{t_{m-1}}^{t_m}\left\Vert e^{(A_h+J^h_{m-1})(t_m-s)}(G^h_{m-1}(X^h(s))-G^h_{m-1}(X^h_{m-1}))\right\Vert_{L^2(\Omega, H)}ds\right)^2\nonumber\\
&\leq&\left( \int_{t_{m-1}}^{t_m}\left(\mathbb{E}\left[\left\Vert e^{(A_h+J^h_{m-1})(t_m-s)}\right\Vert^2_{L(H)} \Vert(G^h_{m-1}(X^h(s))-G^h_{m-1}(X^h_{m-1}))\Vert^2\right]\right)^{1/2}ds\right)^2\nonumber\\
&\leq& C\left(\int_{t_{m-1}}^{t_m}\left(\mathbb{E}\Vert G^h_{m-1}(X^h(s))-G^h_{m-1}(X^h_{m-1})\Vert^2\right)^{1/2}ds\right)^2\nonumber\\
&\leq&C\left(\int_{t_{m-1}}^{t_m}\left(\mathbb{E}\Vert X^h(s)-X^h_{m-1})\Vert^2\right)^{1/2}ds\right)^2=C\left(\int_{t_{m-1}}^{t_m}
\Vert X^h(s)-X^h_{m-1})\Vert_{L^2(\Omega,H)}ds\right)^2\nonumber\\
&\leq& C\left(\int_{t_{m-1}}^{t_m}\Vert X^h(s)-X^h(t_{m-1})\Vert_{L^2(\Omega, H)}ds\right)^2
+C\Delta t^2\Vert X^h(t_{m-1})-X^h_{m-1}\Vert^2_{L^2(\Omega, H)}.
\end{eqnarray}
}}
Using Lemma \ref{lemma4}, it follows from \eqref{num2aaa} that 
\begin{eqnarray}
\label{num2aaaa}
III&\leq& C\left(\int_{t_{m-1}}^{t_m}(s-t_{m-1})^{\beta/2}ds\right)^2+C\Delta t^2\Vert X^h(t_{m-1})-X^h_{m-1}\Vert^2_{L^2(\Omega,H)}\nonumber\\
&\leq& C\Delta t^{2+\beta}+C\Delta t^2\Vert X^h(t_{m-1})-X^h_{m-1}\Vert^2_{L^2(\Omega, H)}.
\end{eqnarray}
Since the estimates of $IV$ and $VI$ are much more complicated, let us estimate $V$ first.
We use  inequality $(a+b)^2\leq 2a^2+2b^2$  to split $V$ into two terms. This yields 
\begin{eqnarray}
\label{num2aa}
V&\leq&2\Vert\sum_{k=0}^{m-2}\int_{t_{m-k-2}}^{t_{m-k-1}}e^{(A_h+J^h_{m-1})\Delta t}\cdots e^{(A_h+J^h_{m-k-1})\Delta t}e^{(A_h+J^h_{m-k-2})(t_{m-k-1}-s)}\nonumber\\
& &(G^h_{m-k-2}(X^h(s))-G^h_{m-k-2}(X^h(t_{m-k-2})))ds\Vert^2_{L^2(\Omega, H)}\nonumber\\
&+&2\Vert\sum_{k=0}^{m-2}\int_{t_{m-k-2}}^{t_{m-k-1}}e^{(A_h+J^h_{m-1})\Delta t}\cdots e^{(A_h+J^h_{m-k-1})\Delta t}e^{(A_h+J^h_{m-k-2})(t_{m-k-1}-s)}\nonumber\\
& &(G^h_{m-k-2}(X^h(t_{m-k-2}))-G^h_{m-k-2}(X^h_{m-k-2}))ds\Vert^2_{L^2(\Omega, H)}\nonumber\\
&=:&2V_1+2V_2.
\end{eqnarray}
Using  triangle inequality gives
\begin{eqnarray}
V_{1}&=&\Vert\sum_{k=0}^{m-2}\int_{t_{m-k-2}}^{t_{m-k-1}}e^{(A_h+J^h_{m-1})\Delta t}\cdots e^{(A_h+J^h_{m-k-1})\Delta t}e^{(A_h+J^h_{m-k-2})(t_{m-k-1}-s)}\nonumber\\
& &(G^h_{m-k-2}(X^h(s))-G^h_{m-k-2}(X^h(t_{m-k-2})))ds\Vert^2_{L^2(\Omega, H)}\nonumber\\
&\leq& m\sum_{k=0}^{m-2}\Vert\int_{t_{m-k-2}}^{t_{m-k-1}} e^{(A_h+J^h_{m-1})\Delta t}\cdots e^{(A_h+J^h_{m-k-1})\Delta t}e^{(A_h+J^h_{m-k-2})(t_{m-k-1}-s)}\nonumber\\
& &(G^h_{m-k-2}(X^h(s))-G^h_{m-k-2}(X^h_{m-k-2}))ds\Vert^2_{L^2(\Omega, H)}\nonumber\\
&\leq&m\sum_{k=0}^{m-2}\left(\int_{t_{m-k-2}}^{t_{m-k-1}}\left(\mathbb{E}\left[\left\Vert e^{(A_h+J^h_{m-1})\Delta t}\cdots e^{(A_h+J^h_{m-k-1})\Delta t}\right\Vert^2_{L(H)}\right.\right.\right.\nonumber\\
&& \left\Vert e^{(A_h+J^h_{m-k-2})(t_{m-k-1}-s)}\right\Vert^2_{L(H)}\nonumber\\
&&\times\left.\left.\left.\left\Vert(G^h_{m-k-2}(X^h(s))-G^h_{m-k-2}(X^h(t_{m-k-2})))\right\Vert^2\right]\right)^{1/2}ds\right)^2.
\end{eqnarray}
Using   \lemref{lemma9} with $\nu=0$ and \lemref{lemma5} yields
{\small
\begin{eqnarray}
\label{num2ab}
V_1\leq Cm\sum_{k=0}^{m-2}\left(\int_{t_{m-k-2}}^{t_{m-k-1}}\left(\mathbb{E}\Vert G^h_{m-k-2}(X^h(s))-G^h_{m-k-2}(X^h(t_{m-k-2}))\Vert^2\right)^{1/2}ds\right)^2.
\end{eqnarray}
}
Using  \lemref{lemma6} and   Lemma \ref{lemma4}, it follows from \eqref{num2ab} that 
\begin{eqnarray}
\label{use1}
V_1&\leq&Cm\sum_{k=0}^{m-2}\left(\int_{t_{m-k-2}}^{t_{m-k-1}}\left(\mathbb{E}\Vert X^h(s)-X^h(t_{m-k-2})\Vert^2\right)^{1/2}ds\right)^2\nonumber\\
&=&Cm\sum_{k=0}^{m-2}\left(\int_{t_{m-k-2}}^{t_{m-k-1}}\Vert X^h(s)-X^h(t_{m-k-2})\Vert_{L^2(\Omega,H)}ds\right)^2\nonumber\\
&\leq& mC\sum_{k=0}^{m-2}\left(\int_{t_{m-k-2}}^{t_{m-k-1}}(s-t_{m-k-2})^{\beta/2}ds\right)^2\leq C\Delta t^{\beta}.
\end{eqnarray}
Using triangle inequality,   Lemma \ref{lemma5}, Lemma \ref{lemma9} with $\nu=0$, Lemma \ref{lemma6} and  Holder's inequality yields
\begin{eqnarray}
\label{use1a}
V_2&=&\Vert\sum_{k=0}^{m-2}\int_{t_{m-k-2}}^{t_{m-k-1}}e^{(A_h+J^h_{m-1})\Delta t}\cdots e^{(A_h+J^h_{m-k-1})\Delta t}e^{(A_h+J^h_{m-k-2})(t_{m-k-1}-s)}\nonumber\\
&&(G_{m-k-2}(X^h(t_{m-k-2}))-G_{m-k-2}(X^h_{m-k-2}))ds\Vert^2_{L^2(\Omega, H)}\nonumber\\
&\leq&m\sum_{k=0}^{m-2}\Vert\int_{t_{m-k-2}}^{t_{m-k-1}}e^{(A_h+J^h_{m-2})\Delta t}\cdots e^{(A_h+J^h_{m-k-1})\Delta t}e^{(A_h+J^h_{m-k-2})(t_{m-k-1}-s)}\nonumber\\
&&(G_{m-k-2}(X^h(t_{m-k-2}))-G_{m-k-2}(X^h_{m-k-2}))ds\Vert^2_{L^2(\Omega, H)}\nonumber\\
&\leq&Cm\Delta t\sum_{k=0}^{m-2}\int_{t_{m-k-2}}^{t_{m-k-1}}\Vert X^h(t_{m-k-2})-X^h_{m-k-2}\Vert^2_{L^2(\Omega, H)}ds\nonumber\\
&\leq&C\sum_{k=0}^{m-2}\Delta t\Vert X^h(t_{m-k-2})-X^h_{m-k-2}\Vert^2_{L^2(\Omega, H)}\nonumber\\
&=&C\Delta t\sum_{k=0}^{m-2}\Vert X^h(t_k)-X^h_k\Vert^2_{L^2(\Omega, H)}.
\end{eqnarray}
Substituting \eqref{use1a} and \eqref{use1} in \eqref{num2aa} yields
\begin{eqnarray}
\label{use1b}
V\leq C\Delta t^{\beta}+C\sum_{k=0}^{m-2}\Delta t\Vert X^h(t_k)-X^h_k\Vert^2_{L^2(\Omega, H)}.
\end{eqnarray}
To estimate $VI$, we use  the triangle inequality to split it in two terms
\begin{eqnarray}
\label{esti4}
&&VI\nonumber\\
&\leq& 2\Vert\sum_{k=0}^{m-2}\int_{t_{m-k-2}}^{t_{m-k-1}}e^{(A_h+J^h_{m-1})\Delta t}\cdots e^{(A_h+J^h_{m-k-1})\Delta t}e^{(A_h+J^h_{m-k-2})(t_{m-k-1}-s)}\nonumber\\
&&\left[P_hB(X^h(s))-P_hB(X^h(t_{m-k-2}))\right]dW(s)\Vert^2_{L^2(\Omega, H)}\nonumber\\
&+&2\Vert\sum_{k=0}^{m-2}\int_{t_{m-k-2}}^{t_{m-k-1}}e^{(A_h+J^h_{m-1})\Delta t}\cdots e^{(A_h+J^h_{m-k-1})\Delta t}\nonumber\\
&&[e^{(A_h+J^h_{m-k-2})(t_{m-k-1}-s)}P_hB(X^h(t_{m-k-2}))-e^{(A_h+J^h_{m-k-2})\Delta t}P_hB(X^h_{m-k-2})]dW(s)\Vert^2_{L^2(\Omega, H)}\nonumber\\
&=:& 2VI_1+2VI_2.
\end{eqnarray}
Since the expectation of the cross-product vanishes, Using Cauchy-Schwartz inequality, it follows that 
\begin{eqnarray*}
&&VI_1\nonumber\\
&=&\mathbb{E}\left[\Vert\sum_{k=0}^{m-2}\int_{t_{m-k-2}}^{t_{m-k-1}}e^{(A_h+J^h_{m-1})\Delta t}\cdots e^{(A_h+J^h_{m-k-1})\Delta t}e^{(A_h+J^h_{m-k-2})(t_{m-k-1}-s)}\right.\nonumber\\
&&\left.\left[P_hB(X^h(s))-P_hB(X^h(t_{m-k-2}))\right]dW(s)\Vert^2\right]\nonumber\\
&=&\sum_{k=0}^{m-2}\mathbb{E}\left[\Vert e^{(A_h+J^h_{m-1})\Delta t}\cdots e^{(A_h+J^h_{m-k-1})\Delta t}e^{(A_h+J^h_{m-k-2})(t_{m-k-1}-s)}\right.\nonumber\\
&&\left.\int_{t_{m-k-2}}^{t_{m-k-1}}\left[P_hB(X^h(s))-P_hB(X^h(t_{m-k-2}))\right]dW(s)\Vert^2\right]\nonumber\\
&\leq&\sum_{k=0}^{m-2}\left(\mathbb{E}\left\Vert e^{(A_h+J^h_{m-1})\Delta t}\cdots e^{(A_h+J^h_{m-k-1})\Delta t}e^{(A_h+J^h_{m-k-2})(t_{m-k-1}-s)}\right\Vert^4_{L(H)}\right)^{\frac{1}{2}}\nonumber\\
&\times&\left(\mathbb{E}\left\Vert\int_{t_{m-k-2}}^{t_{m-k-1}}\left[P_hB(X^h(s))-P_hB(X^h(t_{m-k-2}))\right]dW(s)\right\Vert^4\right)^{\frac{1}{2}}.
\end{eqnarray*}
Using the Burkh\"{o}lder-Davis-Gundy inequality (\cite[Lemma 5.1]{Raphael}),  \lemref{lemma9} with $\nu=0$,  \assref{assumption4} and the fact that $S^h_k(\omega)$ is uniformly bounded (independently of $h$, $k$ and the sample $\omega$)
yields 
\begin{eqnarray}
\label{usemidle}
VI_1&\leq&C\sum_{k=0}^{m-2}\left(\mathbb{E}\left\Vert\int_{t_{m-k-2}}^{t_{m-k-1}}\left[P_hB(X^h(s))-P_hB(X^h(t_{m-k-2}))\right]dW(s)\right\Vert^4\right)^{\frac{1}{2}}\nonumber\\
&\leq&C\sum_{k=0}^{m-2}\int_{t_{m-k-2}}^{t_{m-k-1}}\mathbb{E}\Vert P_hB(X^h(s))-P_hB(X^h(t_{m-k-2}))\Vert^2_{L^0_2}ds\nonumber\\
&\leq&C\sum_{k=0}^{m-2}\int_{t_{m-k-2}}^{t_{m-k-1}}\Vert X^h(s)-X^h(t_{m-k-2})\Vert^2_{L^2(\Omega, H)}ds.
\end{eqnarray}
Applying \lemref{lemma4}, it follows from \eqref{usemidle} that
\begin{eqnarray}
\label{use2}
VI_1\leq C\sum_{k=0}^{m-2}\int_{t_{m-k-2}}^{t_{m-k-1}}(s-t_{m-k-2})^{\beta}ds\leq C\Delta t^{\beta}.
\end{eqnarray}
Using the inequality  $(a+b)^2\leq 2 a^2+2b^2$, we split $VI_2$ in two terms 
\begin{eqnarray}
\label{use2a}
&&VI_2\nonumber\\
&=&\Vert \sum_{k=0}^{m-2}\int_{t_{m-k-2}}^{t_{m-k-1}}e^{(A_h+J^h_{m-1})\Delta t}\cdots e^{(A_h+J^h_{m-k-1})\Delta t}\left[e^{(A_h+J^h_{m-k-2})(t_{m-k-1}-s)}\right.\nonumber\\
&&\left. P_hB(X^h(t_{m-k-2}))-e^{(A_h+J^h_{m-k-2})\Delta t}P_hB(X^h_{m-k-2})\right]dW(s)\Vert^2_{L^2(\Omega, H)}\nonumber\\
&\leq&2\Vert \sum_{k=0}^{m-2}\int_{t_{m-k-2}}^{t_{m-k-1}}e^{(A_h+J^h_{m-1})\Delta t}\cdots e^{(A_h+J^h_{m-k-1})\Delta t}\nonumber\\
&&\left[e^{(A_h+J^h_{m-k-2})(t_{m-k-1}-s)}-e^{(A_h+J^h_{m-k-2})\Delta t}\right]P_hB(X^h(t_{m-k-2}))dW(s)\Vert^2_{L^2(\Omega, H)}\nonumber\\
&+&2\Vert\sum_{k=0}^{m-2}\int_{t_{m-k-2}}^{t_{m-k-1}}e^{(A_h+J^h_{m-1})\Delta t}\cdots e^{(A_h+J^h_{m-k-1})\Delta t} e^{(A_h+J^h_{m-k-2})\Delta t}\nonumber\\
&&\left[P_hB(X^h(t_{m-k-2}))-P_hB(X^h_{m-k-2})\right]dW(s)\Vert^2_{L^2(\Omega, H)}\nonumber\\
&= :& 2VI_{21}+2VI_{22}.
\end{eqnarray}
Since the expectation of the cross-product vanishes, inserting an appropriate power of $A_h$ and using Cauchy-Schwartz inequality yields
{\small
\begin{eqnarray*}
&&VI_{21}\nonumber\\
&=&\mathbb{E}\left[\Vert \sum_{k=0}^{m-2}\int_{t_{m-k-2}}^{t_{m-k-1}}e^{(A_h+J^h_{m-1})\Delta t}\cdots e^{(A_h+J^h_{m-k-1})\Delta t}\right.\nonumber\\
&&\left.\left[e^{(A_h+J^h_{m-k-2})(t_{m-k-1}-s)}-e^{(A_h+J^h_{m-k-2})\Delta t}\right]P_hB(X^h(t_{m-k-2}))dW(s)\Vert^2\right]\nonumber\\
&=& \sum_{k=0}^{m-2}\mathbb{E}\left[\Vert e^{(A_h+J^h_{m-1})\Delta t}\cdots e^{(A_h+J^h_{m-k-1})\Delta t}\right.\nonumber\\
&&\left.\int_{t_{m-k-2}}^{t_{m-k-1}}\left[e^{(A_h+J^h_{m-k-2})(t_{m-k-1}-s)}-e^{(A_h+J^h_{m-k-2})\Delta t}\right]P_hB(X^h(t_{m-k-2}))dW(s)\Vert^2\right]\nonumber\\
&=& \sum_{k=0}^{m-2}\mathbb{E}\left[\Vert e^{(A_h+J^h_{m-1})\Delta t}\cdots e^{(A_h+J^h_{m-k-1})\Delta t}(-A_h)^{\frac{\beta}{2}}\right.\nonumber\\
&&\left.\int_{t_{m-k-2}}^{t_{m-k-1}}(-A_h)^{-\frac{\beta}{2}}\left[e^{(A_h+J^h_{m-k-2})(t_{m-k-1}-s)}-e^{(A_h+J^h_{m-k-2})\Delta t}\right]P_hB(X^h(t_{m-k-2}))dW(s)\Vert^2\right]\nonumber\\
&\leq& \sum_{k=0}^{m-2}\left(\mathbb{E}\Vert e^{(A_h+J^h_{m-1})\Delta t}\cdots e^{(A_h+J^h_{m-k-1})\Delta t}(-A_h)^{\frac{\beta}{2}}\Vert^4_{L(H)}\right)^{\frac{1}{2}}\nonumber\\
&\times&\left(\mathbb{E}\Vert\int_{t_{m-k-2}}^{t_{m-k-1}}(-A_h)^{-\frac{\beta}{2}}\left[e^{(A_h+J^h_{m-k-2})(t_{m-k-1}-s)}-e^{(A_h+J^h_{m-k-2})\Delta t}\right]P_hB(X^h(t_{m-k-2}))dW(s)\Vert^4\right)^{\frac{1}{2}}. 
\end{eqnarray*}
}
Using the Burkh\"{o}lder-Davis-Gundy inequality  yields
{\small
\begin{eqnarray}
\label{use2b}
&&VI_{21}\\
&\leq&\sum_{k=0}^{m-2}\left(\mathbb{E}\Vert e^{(A_h+J^h_{m-1})\Delta t}\cdots e^{(A_h+J^h_{m-k-1})\Delta t}(-A_h)^{\frac{\beta}{2}}\Vert^4_{L(H)}\right)^{\frac{1}{2}}\nonumber\\
&\times&\int_{t_{m-k-2}}^{t_{m-k-1}}\mathbb{E}\Vert(-A_h)^{-\frac{\beta}{2}}\left[e^{(A_h+J^h_{m-k-2})(t_{m-k-1}-s)}-e^{(A_h+J^h_{m-k-2})\Delta t}\right]P_hB(X^h(t_{m-k-2}))\Vert^2_{L^0_2(H)}ds. \nonumber
\end{eqnarray}
}
As  $S^h_k(t)$ is a semigroup, we obviously have
\begin{eqnarray}
\label{smoothnew}
 S^h_k(t+s)=S^h_k(t)S^h_k(s), \quad t,s\geq 0.
\end{eqnarray}
Using relation \eqref{smoothnew},  \lemref{lemma8} (i)  with $\gamma_1=\frac{\beta}{2}$ and $\gamma_2=0$ and  \lemref{lemma9} with $\nu=\frac{\beta}{2}$  in \eqref{use2b} allows to have 
\begin{eqnarray}
\label{case1a}
VI_{21}
&\leq&\sum_{k=0}^{m-2}\int_{t_{m-k-2}}^{t_{m-k-1}}\left(\mathbb{E}\left\Vert e^{(A_h+J^h_{m-1})\Delta t}\cdots e^{(A_h+J^h_{m-k-1})\Delta t}(-A_h)^{\frac{\beta}{2}}\right\Vert^4_{L(H)}\right)^{\frac{1}{2}}\nonumber\\
&&\times\mathbb{E}\left[\Vert (-A_h)^{\frac{-\beta}{2}}\left(\mathbf{I}-S^h_{m-k-2}(s-t_{m-k-2})\right)\Vert^2_{L(H)}\Vert S^h_{m-k-2}(t_{m-k-1}-s)\Vert^2_{L(H)}\right.\nonumber\\
&&\times \left.\Vert P_hB(X^h(t_{m-k-2}))\Vert^2_{L^0_2}\right]ds\\
&\leq& C\sum_{k=0}^{m-2}\int_{t_{m-k-2}}^{t_{m-k-1}}t_{k+1}^{-\beta}(s-t_{m-k-2})^{\beta}\Vert S^h_{m-k-2}(t_{m-k-1}-s)\Vert^2_{L(H)}\\
&&\times\Vert B(X^h(t_{m-k-2}))\Vert^2_{L^0_2}ds\nonumber.
\end{eqnarray}
  Using \assref{assumption4} and the fact that the random perturbed semigroup $S^h_{k}$ is uniformly bounded independently of $k$, $h$ and the sample $\omega$, 
  it follows from \eqref{case1a} that
  \begin{eqnarray}
  \label{case1b}
  VI_{21}&\leq& C\sum_{k=0}^{m-2}\int_{t_{m-k-2}}^{t_{m-k-1}}t_{k+1}^{-\beta}(s-t_{m-k-2})^{\beta}ds\nonumber\\
  &\leq& C\Delta t^{\beta}\sum_{k=0}^{m-2}t_{k+1}^{-\beta}\Delta t\leq C\Delta t^{\beta}.
  \end{eqnarray}
Let us estimate $VI_{22}$
\begin{eqnarray}
VI_{22}&: =&\left\Vert\sum_{k=0}^{m-2}\int_{t_{m-k-2}}^{t_{m-k-1}}e^{(A_h+J^h_{m-1})\Delta t}\cdots e^{(A_h+J^h_{m-k-1})\Delta t}e^{(A_h+J^h_{m-k-2})\Delta t}\right.\nonumber\\
&&\left.\left[P_hB(X^h(t_{m-k-2}))-P_hB(X^h_{m-k-2})\right]dW(s)\right\Vert^2_{L^2(\Omega. H)}.\nonumber\\
\end{eqnarray}
Following the same lines  as in the estimate of $VI_1$, the following estimate holds for $VI_{22}$
\begin{eqnarray}
\label{case1c}
VI_{22} &\leq&C\sum_{k=0}^{m-2}\int_{t_{m-k-2}}^{t_{m-k-1}}\Vert X^h(t_{m-k-2})-X^h_{m-k-2}\Vert^2_{L^2(\Omega, H)}ds\nonumber\\
&\leq& C\sum_{k=0}^{m-2}\Delta t\Vert X^h(t_{m-k-2})-X^h_{m-k-2}\Vert^2_{L^2(\Omega, H)}\nonumber\\
&=&C\Delta t\sum_{k=0}^{m-2}\Vert X^h(t_k)-X^h_k\Vert^2_{L^2(\Omega, H)}.
\end{eqnarray}
Inserting \eqref{case1c} and \eqref{case1b} in \eqref{use2a} gives 
\begin{eqnarray}
\label{case1d}
VI_2\leq C\Delta t^{\beta}+C\Delta t\sum_{k=0}^{m-2}\Vert X^h(t_k)-X^h_k\Vert^2_{L^2(\Omega, H)}.
\end{eqnarray}
Substituting  \eqref{use2} and \eqref{case1d}  in \eqref{esti4} yields 
\begin{eqnarray}
\label{more4c}
VI\leq C\Delta t^{\beta}+C\Delta t\sum_{k=0}^{m-2}\Vert X^h(t_k)-X^h_k\Vert^2_{L^2(\Omega, H)}.
\end{eqnarray}
Following the same lines as for the estimate of  $VI$, we obtain 
\begin{eqnarray}
\label{more4d}
IV&=&\left\Vert\int_{t_{m-1}}^{t_m}\left[e^{(A_h+J^h_{m-1})(t_m-s)}P_hB(X^h(s))\right.\right.\nonumber\\
&&\left.\left.-e^{(A_h+J^h_{m-1})\Delta t}P_hB(X^h_{m-1})\right]dW(s)\right\Vert^2_{L^2(\Omega, H)}\nonumber\\
&\leq& C\Delta t^{\beta}+C\Delta t\Vert X^h(t_{m-1})-X^h_{m-1}\Vert^2_{L^2(\Omega, H)}.
\end{eqnarray}
Gathering estimates of $III$, $IV$, $V$ and $VI$  in \eqref{somme1} yields  
\begin{eqnarray}
\label{above1}
\Vert X^h(t_m)-X^h_m\Vert^2_{L^2(\Omega, H)}&\leq& C\Delta t^{\beta}+C\Delta t\sum_{k=0}^{m-1}\Vert X^h(t_k)-X^h_k\Vert^2_{L^2(\Omega, H)}.
\end{eqnarray}
Applying the discrete Gronwall lemma to \eqref{above1}  yields 
\begin{eqnarray}
\label{final}
\Vert X^h(t_m)-X^h_m\Vert_{L^2(\Omega, H)}\leq  C\Delta t^{\beta/2}.
\end{eqnarray}
 Using \lemref{lemma7} together with the estimate \eqref{final} completes the proof of  \thmref{mainresult1} for $0\leq\beta<1$.
\subsubsection{Estimate of the time error for $1\leq \beta\leq 2$}
\label{secondestimate}
Note that the estimates of $III$ and $V$  in Section \ref{firstestimate} hold for $\beta\in[1,2]$ and due to the time regularity in \lemref{lemma4}, we obtain from \eqref{use1b} and \eqref{num2aaaa}
\begin{eqnarray}
\label{stand1}
III+V\leq C\Delta t+C\Delta t \sum_{k=0}^{m-1}\Vert X^h(t_k)-X^h_k\Vert^2_{L^2(\Omega,H)}.
\end{eqnarray} 
We only need to re-estimate $IV$ and $VI$. We will only estimate $VI$ in details since the the estimate of $IV$ is similar to that of $VI$. Let us recall that using triangle inequality we obtain
\begin{eqnarray}
\label{new1}
VI\leq 2VI_1+2VI_2,
\end{eqnarray}
 where $VI_1$ and $VI_2$ are defined by \eqref{esti4} in Section \ref{firstestimate}. Applying \lemref{lemma4}, it follows from \eqref{usemidle} that 
 \begin{eqnarray}
 \label{new2}
 VI_1\leq C\sum_{k=0}^{m-2}\int_{t_{m-k-2}}^{t_{m-k-1}}(s-t_{m-k-2})ds\leq C\Delta t.
 \end{eqnarray}
 Using   the  inequality  $(a+b)^2\leq 2 a^2+2b^2$, we split $VI_2$ in two terms 
\begin{eqnarray}
\label{case2a}
VI_2\leq VI_{21}+VI_{22},
\end{eqnarray}
where $VI_{21}$ and $VI_{22}$ are given by \eqref{use2a} in Subsection \ref{firstestimate}. We recall that from \eqref{case1c} the following estimate holds for $VI_{22}$
\begin{eqnarray}
\label{case2aa}
VI_{22}\leq C\Delta t\sum_{k=0}^{m-2}\Vert X^h(t_k)-X^h_k\Vert^2_{L^2(\Omega,H)}.
\end{eqnarray}
Since the expectation of the cross product vanishing, inserting an appropriate power of $A_h$ and using Cauchy-Schwartz inequality yields
{\small
\begin{eqnarray*}
&&VI_{21}\nonumber\\
&=&\mathbb{E}\left\Vert\sum_{k=0}^{m-2}\int_{t_{m-k-2}}^{t_{m-k-1}} e^{(A_h+J^h_{m-1})\Delta t}\cdots e^{(A_h+J^h_{m-k-1})\Delta t}\right.\nonumber\\
&&\left.\left[e^{(A_h+J^h_{m-k-2})(t_{m-k-1}-s)}-e^{(A_h+J^h_{m-k-2})\Delta t}\right] P_hB(X^h(t_{m-k-2}))dW(s)\right\Vert^2\nonumber\\
&=&\sum_{k=0}^{m-2}\mathbb{E}\left\Vert e^{(A_h+J^h_{m-1})\Delta t}\cdots e^{(A_h+J^h_{m-k-1})\Delta t}\right.\nonumber\\
&&\left.\int_{t_{m-k-2}}^{t_{m-k-1}}\left[e^{(A_h+J^h_{m-k-2})(t_{m-k-1}-s)}-e^{(A_h+J^h_{m-k-2})\Delta t}\right] P_hB(X^h(t_{m-k-2}))dW(s)\right\Vert^2\nonumber\\
&=&\sum_{k=0}^{m-2}\mathbb{E}\left\Vert e^{(A_h+J^h_{m-1})\Delta t}\cdots e^{(A_h+J^h_{m-k-1})\Delta t}(-A_h)^{\frac{1-\gamma}{2}}\right.\nonumber\\
&&\left.\int_{t_{m-k-2}}^{t_{m-k-1}}(-A_h)^{\frac{-1+\gamma}{2}}\left[e^{(A_h+J^h_{m-k-2})(t_{m-k-1}-s)}-e^{(A_h+J^h_{m-k-2})\Delta t}\right] P_hB(X^h(t_{m-k-2}))dW(s)\right\Vert^2\nonumber\\
&\leq&\sum_{k=0}^{m-2}\left(\mathbb{E}\left\Vert e^{(A_h+J^h_{m-1})\Delta t}\cdots e^{(A_h+J^h_{m-k-1})\Delta t}(-A_h)^{\frac{1-\gamma}{2}}\right\Vert^4_{L(H)}\right)^{\frac{1}{2}}\nonumber\\
&\times&\left(\mathbb{E}\left\Vert\int_{t_{m-k-2}}^{t_{m-k-1}}(-A_h)^{\frac{-1+\gamma}{2}}\left[e^{(A_h+J^h_{m-k-2})(t_{m-k-1}-s)}-e^{(A_h+J^h_{m-k-2})\Delta t}\right] P_hB(X^h(t_{m-k-2}))dW(s)\right\Vert^4\right)^{\frac{1}{2}}.
\end{eqnarray*}
}
Using  Burkh\"{o}lder-Davis-Gundy inequality and  the triangle inequality,   we  split $VI_{21}$ in two parts as 
{\small
\begin{eqnarray}
\label{case2b}
&&VI_{21}\nonumber\\
&\leq&\sum_{k=0}^{m-2}\int_{t_{m-k-2}}^{t_{m-k-1}}\left(\mathbb{E}\left\Vert e^{(A_h+J^h_{m-1})\Delta t}\cdots e^{(A_h+J^h_{m-k-1})\Delta t}(-A_h)^{\frac{1-\gamma}{2}}\right\Vert^4_{L(H)}\right)^{\frac{1}{2}}\nonumber\\
&\times&\int_{t_{m-k-2}}^{t_{m-k-1}}\mathbb{E}\left\Vert(-A_h)^{\frac{-1+\gamma}{2}}\left[e^{(A_h+J^h_{m-k-2})(t_{m-k-1}-s)}-e^{(A_h+J^h_{m-k-2})\Delta t}\right] P_hB(X^h(t_{m-k-2}))\right\Vert^2_{L^0_2}ds\nonumber\\
&\leq& 2\sum_{k=0}^{m-1}\left(\mathbb{E}\left\Vert e^{(A_h+J^h_{m-1})\Delta t}\cdots e^{(A_h+J^h_{m-k-1})\Delta t}(-A_h)^{\frac{1-\gamma}{2}}\right\Vert^4_{L(H)}\right)^{\frac{1}{2}}\nonumber\\
&\times&\int_{t_{m-k-2}}^{t_{m-k-1}}\mathbb{E}\left\Vert\left[e^{(A_h+J^h_{m-k-2})(t_{m-k-1}-s)}-e^{(A_h+J^h_{m-k-2})\Delta t}\right]\left[P_hB(X^h(t_{m-k-2}))-P_hB(X(t_{m-k-2}))\right]\right\Vert^2_{L^0_2}ds\nonumber\\
&+&2\sum_{k=0}^{m-2} \left(\mathbb{E}\left\Vert e^{(A_h+J^h_{m-1})\Delta t}\cdots e^{(A_h+J^h_{m-k-1})\Delta t}(-A_h)^{\frac{1-\gamma}{2}}\right\Vert^4_{L(H)}\right)^{\frac{1}{2}}\nonumber\\
&\times&\int_{t_{m-k-2}}^{t_{m-k-1}}\mathbb{E}\left\Vert\left[e^{(A_h+J^h_{m-k-2})(t_{m-k-1}-s)}-e^{(A_h+J^h_{m-k-2})\Delta t}\right]P_hB(X(t_{m-k-2}))\right\Vert^2_{L^0_2}ds\nonumber\\
&:=& 2VI_{211}+2VI_{212}.
\end{eqnarray}
}
Using \lemref{lemma5}  and \assref{assumption4} yields
{ \small {
\begin{eqnarray}
\label{case2c}
&&VI_{211}\\
&\leq&C\sum_{k=0}^{m-2}\left(\mathbb{E}\left\Vert e^{(A_h+J^h_{m-1})\Delta t}\cdots e^{(A_h+J^h_{m-k-1})\Delta t}(-A_h)^{\frac{1-\gamma}{2}}\right\Vert^4_{L(H)}\right)^{\frac{1}{2}}\nonumber\\
&\times&\int_{t_{m-k-2}}^{t_{m-k-1}}\left[\left\Vert(-A_h)^{\frac{-1+\gamma}{2}}\left( e^{(A_h+J^h_{m-k-2})(t_{m-k-1}-s)}-e^{(A_h+J^h_{m-k-2})\Delta t}\right)\right\Vert^2_{L(H)}\Vert X(t_{m-k-2})-X^h(t_{m-k-2})\Vert^2\right]ds.\nonumber
\end{eqnarray}
}}
 Using \lemref{lemma7} and \lemref{lemma9} with $\nu=0$, we obtain
  \begin{eqnarray}
  \label{case2c}
 VI_{211}&\leq& C\sum_{k=0}^{m-2}t_{k+1}^{-1+\gamma}\int_{t_{m-k-2}}^{t_{m-k-1}}\mathbb{E}\Vert X(t_{m-k-2})-X^h(t_{m-k-2})\Vert^2ds\nonumber\\
 &\leq& Ch^{2\beta}.
\end{eqnarray}
Inserting an appropriated power of $-A_h$ yields the following estimate
\begin{eqnarray}
\label{case2d}
&&VI_{212}\nonumber\\
&\leq&\sum_{k=0}^{m-2}\left(\mathbb{E}\Vert e^{(A_h+J^h_{m-1})\Delta t}\cdots e^{(A_h+J^h_{m-k-1})\Delta t}(-A_h)^{\frac{1-\gamma}{2}}\Vert^4_{L(H)}\right)^{\frac{1}{2}}\nonumber\\
&\times&\int_{t_{m-k-2}}^{t_{m-k-1}}\mathbb{E}\left[\Vert(-A_h)^{\frac{-1+\gamma}{2}}\left(e^{(A_h+J^h_{m-k-2})(t_{m-k-1}-s)}-e^{(A_h+J^h_{m-k-2})\Delta t}\right)(-A_h)^{\frac{-\gamma}{2}}\Vert^2_{L(H)}\right.\nonumber\\
&&\times\left.\Vert (-A_h)^{\frac{\gamma}{2}}P_hB(X(t_{m-k-2}))\Vert^2_{L^0_2}\right]ds.
\end{eqnarray}
Using \lemref{lemma9} with $\nu=\frac{1-\gamma}{2}$ in \eqref{case2d}, we obatin  
\begin{eqnarray}
 \label{case2e}
 &&VI_{212}\nonumber\\
 &\leq&C\sum_{k=0}^{m-2}\int_{t_{m-k-2}}^{t_{m-k-1}}t_{k+1}^{-1+\gamma}\mathbb{E}\left[\Vert(-A_h)^{\frac{1+\gamma}{2}}\left(e^{(A_h+J^h_{m-k-2})(t_{m-k-1}-s)}-e^{(A_h+J^h_{m-k-2})\Delta t}\right)\right.\nonumber\\
 &&\times\left.(-A_h)^{\frac{-\gamma}{2}}\Vert^2_{L(H)}\Vert(-A_h)^{\frac{\gamma}{2}}P_hB(X(t_{m-k-2}))\Vert^2_{L^0_2}\right]ds\nonumber\\
 &\leq& C\sum_{k=0}^{m-2}t_{k+1}^{-1+\gamma}\int_{t_{m-k-2}}^{t_{m-k-1}}\mathbb{E}\left[\left\Vert(-A_h)^{\frac{-1+\gamma}{2}}\left[e^{(A_h+J^h_{m-k-2})(t_{m-k-1}-s)}\right.\right.\right.\nonumber\\
&&\left.\left.\left.-e^{(A_h+J^h_{m-k-2})\Delta t}\right](-A_h)^{-\frac{\gamma}{2}}\right\Vert^2_{L(H)} \Vert(-A_h)^{\frac{\gamma}{2}}P_hB(X(t_{m-k-2}))\Vert^2_{L^0_2}\right]ds.\nonumber\\
\end{eqnarray}
Using relation \eqref{smoothnew}, \lemref{lemma8} (i) with $\gamma_1=\frac{1-\gamma}{2}$ and $\gamma_2=\frac{\gamma}{2}$  yields 
\begin{eqnarray}
\label{case2ea}
VI_{212}&\leq& C\sum_{k=0}^{m-2}t_{k+1}^{-1+\gamma}\int_{t_{m-k-2}}^{t_{m-k-1}}\mathbb{E}\left[\Vert (-A_h)^{\frac{-1+\gamma}{2}} S^h_{m-k-2}(t_{m-k-1}-s)\right.\nonumber\\
&&\left.(S^h_{m-k-2}(s-t_{m-k-2})-\mathbf{I})(-A_h)^{-\frac{\gamma}{2}}\Vert^2_{L(H)}\Vert(-A_h)^{\frac{\gamma}{2}}P_hB(X(t_{m-k-2}))\Vert^2_{L^0_2}\right] ds\nonumber\\
&\leq& C\sum_{k=0}^{m-2}\left(t_{k+1}^{-1+\gamma}\int_{t_{m-k-2}}^{t_{m-k-1}}\mathbb{E}\left[\left\Vert (-A_h)^{\frac{-1+\gamma}{2}} S^h_{m-k-2}(t_{m-k-1}-s)(-A_h)^{\frac{1-\gamma}{2}}\right\Vert^2_{L(H)}\right.\right.\nonumber\\
&\times&\left\Vert(-A_h)^{\frac{-1+\gamma}{2}}(S^h_{m-k-2}(s-t_{m-k-2})-\mathbf{I})(-A_h)^{-\frac{\gamma}{2}}\right\Vert^2_{L(H)}\nonumber\\
&&\left.\left.\Vert(-A_h)^{\frac{\gamma}{2}}P_hB(X(t_{m-k-2}))\Vert^2_{L^0_2}\right] ds\right)\nonumber\\
&\leq&C\sum_{k=0}^{m-2}t^{-1+\gamma}_{k+1}\int_{t_{m-k-2}}^{t_{m-k-1}}(s-t_{m-k-2})\mathbb{E}\Vert(-A_h)^{\frac{\gamma}{2}}P_hB(X(t_{m-k-2}))\Vert^2_{L^0_2} ds\nonumber\\
&\leq&C\Delta t^{2}\sum_{k=0}^{m-2}t_{k+1}^{-1+\gamma}\mathbb{E}\Vert(-A_h)^{\frac{\gamma}{2}}P_hB(X(t_{m-k-2}))\Vert^2_{L^0_2}.
 \end{eqnarray}
Using the definition of the $L^0_2$ norm, \assref{assumption5}, \lemref{lemma1}, \thmref{theorem1} and   estimate \eqref{regular1a}, we obtain
\begin{eqnarray}
\label{case2f}
\mathbb{E}\Vert (-A_h)^{\frac{\gamma}{2}}P_hB(X(t_{m-k-2}))\Vert^2_{L^0_2}&=&\mathbb{E}\left[\sum_{i=0}^{\infty}\Vert (-A_h)^{\frac{\gamma}{2}}P_hB(X(t_{m-k-2}))Q^{1/2}e_i\Vert^2\right]\nonumber\\
&\leq&C\mathbb{E}\left[\sum_{i=0}^{\infty}\Vert (-A)^{\frac{\gamma}{2}}B(X(t_{m-k-2}))Q^{1/2}e_i\Vert^2\right]\nonumber\\
&=&C\mathbb{E}\Vert (-A)^{\frac{\gamma}{2}}B(X(t_{m-k-2})\Vert^2_{L^0_2}\nonumber\\
&\leq& C\mathbb{E}(1+\Vert (-A)^{\frac{\gamma}{2}}X(t_{m-k-2})\Vert^2)\nonumber\\
&\leq& C(1+\mathbb{E}(\Vert X_0\Vert^2_{\gamma}))<\infty.
\end{eqnarray}
Substituting \eqref{case2f} in \eqref{case2ea} yields
\begin{eqnarray}
\label{case2g}
VI_{212}\leq C\Delta t^{2}\sum_{k=0}^{m-2}t_{k+1}^{-1+\gamma}\leq C\Delta t.
\end{eqnarray}
Inserting \eqref{case2g} and \eqref{case2c} in \eqref{case2b} gives 
\begin{eqnarray}
\label{case2h}
VI_{21}\leq Ch^{2\beta}+C\Delta t.
\end{eqnarray}
Inserting \eqref{case2h} and \eqref{case2aa} in \eqref{case2a} gives 
\begin{eqnarray}
\label{case2i}
VI_2\leq Ch^{2\beta}+C\Delta t+C\Delta t\sum_{k=0}^{m-2}\Vert X^h(t_k)-X^h_k\Vert^2_{L^2(\Omega, H)}.
\end{eqnarray}
Substituting estimates of $VI_2$ \eqref{case2i} and $VI_1$ \eqref{new2}  in \eqref{new1} yields 
\begin{eqnarray}
\label{case2j}
VI\leq Ch^{2\beta}+C\Delta t+C\Delta t\sum_{k=0}^{m-2}\Vert X^h(t_k)-X^h_k\Vert^2_{L^2(\Omega, H)}.
\end{eqnarray}
Following the same lines as for $VI$, we obtain 
\begin{eqnarray}
\label{case2k}
IV&=&\left\Vert\int_{t_{m-1}}^{t_m}\left[e^{(A_h+J^h_{m-1})(t_m-s)}P_hB(X^h(s))\right.\right.\nonumber\\
&&\left.\left.-e^{(A_h+J^h_{m-1})\Delta t}P_hB(X^h_{m-1})\right]dW(s)\right\Vert^2_{L^2(\Omega, H)}\nonumber\\
&\leq& Ch^{2\beta}+C\Delta t+C\Delta t\Vert X^h(t_{m-1})-X^h_{m-1}\Vert^2_{L^2(\Omega, H)}.
\end{eqnarray}
Substituting \eqref{stand1}, \eqref{case2j} and \eqref{case2k} in   \eqref{somme1} yields 
\begin{eqnarray}
\label{case2ka}
&&\Vert X^h(t_m)-X^h_m\Vert^2_{L^2(\Omega,H)}\nonumber\\
&\leq& Ch^{2\beta}+C\Delta t+C\Delta t\sum_{k=0}^{m-1}\Vert X^h(t_k)-X^h_k\Vert^2_{L^2(\Omega,H)}.
\end{eqnarray}
Applying the discrete Gronwall lemma to \eqref{case2ka} yields
\begin{eqnarray}
\label{case2m}
\Vert X^h(t_m)-X^h_m\Vert_{L^2(\Omega,H)}\leq Ch^{\beta}+C\Delta t^{1/2}.
\end{eqnarray}

\subsection{Proof of \thmref{mainresult2}}
\label{theofinal}
Recall that we only need to estimate the time error since the space error is estimated in \lemref{lemma7}. Recall also that the time error can be recast as follow 
\begin{eqnarray}
\label{decom1}
\frac{1}{4}\Vert X^h(t_m)-X^h_m\Vert^2_{L^2(\Omega, H)}\leq III+IV+V+VI,
\end{eqnarray}
where $III$ and $V$ the same as in Section \ref{firstestimate}.
 The terms involving the noise $ IV$ and $VI$ are in this case given by
\begin{eqnarray}
\label{noise1}
IV=\left\Vert\int_{t_{m-1}}^{t_m}\left(e^{(A_h+J^h_{m-1})(t_m-s)}-e^{(A_h+J^h_{m-1})\Delta t}\right)P_hdW(s)\right\Vert^2_{L^2(\Omega, H)}
\end{eqnarray}
and
\begin{eqnarray}
\label{noise2}
VI&=&\left\Vert\sum_{k=0}^{m-2}\int_{t_{m-k-2}}^{t_{m-k-1}}e^{(A_h+J^h_{m-1})\Delta t}\cdots e^{(A_h+J^h_{m-k-1})\Delta t}\right.\nonumber\\
&&\left.\left(e^{(A_h+J^h_{m-k-2})(t_{m-k-1}-s)}-e^{(A_h+J^h_{m-k-1})\Delta t}\right)P_hdW(s)\right\Vert^2_{L^2(\Omega, H)}.
\end{eqnarray}
Recall that from  \eqref{num2aaaa}  we have
\begin{eqnarray}
\label{decom2}
III\leq C\Delta t^{2}+C\Delta t^2\Vert X^h(t_{m-1})-X^h_{m-1}\Vert^2_{L^2(\Omega, H)}.
\end{eqnarray}
It remains to estimate $V$, $IV$ and $VI$. Let us recall that from \eqref{num2aa} we have 
\begin{eqnarray}
\label{revi1}
V\leq 2V_1+2V_2,
\end{eqnarray}
where from \eqref{use1a} we have 
\begin{eqnarray}
\label{revi1a}
V_2\leq C\Delta t\sum_{k=0}^{m-2}\Vert X^h(t_k)-X^h_k\Vert^2_{L^2(\Omega,H)}.
\end{eqnarray}
Let us recall that from \eqref{num2aa} we have 
\begin{eqnarray}
\label{revi2}
\sqrt{V_1}&=& \left\Vert\sum_{k=0}^{m-2}\int_{t_{m-k-2}}^{t_{m-k-1}}e^{(A_h+J^h_{m-1})\Delta t}\cdots e^{(A_h+J^h_{m-k-1})\Delta t}e^{(A_h+J^h_{m-k-2})(t_{m-k-1}-s)}\right.\nonumber\\
&&.\left.(G^h_{m-k-2}(X^h(s))-G^h_{m-k-2}(X^h(t_{m-k-2})))ds\right\Vert_{L^2(\Omega,H)}.
\end{eqnarray}
Using the Taylor formula in Banach space yields 
\begin{eqnarray}
\label{revi3}
\sqrt{V_1}&\leq& \left\Vert\sum_{k=0}^{m-2}\int_{t_{m-k-2}}^{t_{m-k-1}}e^{(A_h+J^h_{m-1})\Delta t}\cdots e^{(A_h+J^h_{m-k-1})\Delta t}e^{(A_h+J^h_{m-k-2})(t_{m-k-1}-s)}\right.\nonumber\\
&&.\left.(G^h_{m-k-2})'(X^h(t_{m-k-2}))\left(e^{(A_h+J^h_{m-k-2})(s-t_{m-k-2})}-\mathbf{I}\right)X^h(t_{m-k-2})ds\right\Vert_{L^2(\Omega,H)}\nonumber\\
&+& \left\Vert\sum_{k=0}^{m-2}\int_{t_{m-k-2}}^{t_{m-k-1}}e^{(A_h+J^h_{m-1})\Delta t}\cdots e^{(A_h+J^h_{m-k-1})\Delta t}e^{(A_h+J^h_{m-k-2})(t_{m-k-1}-s)}\right.\nonumber\\
&&\left.(G^h_{m-k-2})'(X^h(t_{m-k-2}))\int_{t_{m-k-2}}^se^{(A_h+J^h_{m-k-2})(s-\sigma)}G^h_{m-k-2}(X^h(\sigma))d\sigma ds\right\Vert_{L^2(\Omega,H)}\nonumber\\
&+& \left\Vert\sum_{k=0}^{m-2}\int_{t_{m-k-2}}^{t_{m-k-1}}e^{(A_h+J^h_{m-1})\Delta t}\cdots e^{(A_h+J^h_{m-k-1})\Delta t}e^{(A_h+J^h_{m-k-2})(t_{m-k-1}-s)}\right.\nonumber\\
&&.\left.(G^h_{m-k-2})'(X^h(t_{m-k-2}))\int_{t_{m-k-2}}^se^{(A_h+J^h_{m-k-2})(s-\sigma)}P_hdW(\sigma)ds\right\Vert_{L^2(\Omega,H)}\nonumber\\
&+& \left\Vert\sum_{k=0}^{m-2}\int_{t_{m-k-2}}^{t_{m-k-1}}e^{(A_h+J^h_{m-1})\Delta t}\cdots e^{(A_h+J^h_{m-k-1})\Delta t}e^{(A_h+J^h_{m-k-2})(t_{m-k-1}-s)}R_{G^h_{m-k-2}}\right\Vert_{L^2(\Omega,H)}\nonumber\\
&:=&\sqrt{V_{11}}+\sqrt{V_{12}}+\sqrt{V_{13}}+\sqrt{V_{14}},
\end{eqnarray}
where 
{\small
\begin{eqnarray}
\label{revi4}
R_{G^h_{m-k-2}}:=\int_0^1(G^h_{m-k-2})''\left(X^h(t_{m-k-2})+\lambda(X^h(s)-X^h(t_{m-k-2}))\right)\left(X^h(s)-X^h(t_{m-k-2}), X^h(s)-X^h(t_{m-k-2})\right)(1-\lambda)d\lambda.\nonumber
\end{eqnarray}
}
Using  triangle inequality yields 
{\small
\begin{eqnarray}
\label{revi5}
\sqrt{V_{11}}&\leq&\sum_{k=0}^{m-2}\int_{t_{m-k-2}}^{t_{m-k-1}}\left\Vert e^{(A_h+J^h_{m-1})\Delta t}\cdots e^{(A_h+J^h_{m-k-1})\Delta t}e^{(A_h+J^h_{m-k-2})(t_{m-k-1}-s)}\right.\\
&&.\left.(G^h_{m-k-2})'(X^h(t_{m-k-2}))\left(e^{(A_h+J^h_{m-k-2})(s-t_{m-k-2})}-\mathbf{I}\right)X^h(t_{m-k-2})\right\Vert_{L^2(\Omega,H)}ds\nonumber\\
&\leq& \sum_{k=0}^{m-2}\int_{t_{m-k-2}}^{t_{m-k-1}}\left(\mathbf{E}\left[\left\Vert e^{(A_h+J^h_{m-1})\Delta t}\cdots e^{(A_h+J^h_{m-k-1})\Delta t}e^{(A_h+J^h_{m-k-2})(t_{m-k-1}-s)}\right\Vert^2_{L(H)}\right.\right.\nonumber\\
&&\times\left.\left.\left\Vert(G^h_{m-k-2})'(X^h(t_{m-k-2}))\left(e^{(A_h+J^h_{m-k-2})(s-t_{m-k-2})}-\mathbf{I}\right)X^h(t_{m-k-2})\right\Vert^2\right]\right)^{1/2}ds.\nonumber
\end{eqnarray}
}
Using \lemref{lemma9} with $\nu=0$ and \lemref{lemma8} (ii) with $\gamma_1=0$, it holds that
\begin{eqnarray}
\label{revi6}
&&\left\Vert e^{(A_h+J^h_{m-1}(\omega))\Delta t}\cdots e^{(A_h+J^h_{m-k-1}(\omega))\Delta t}e^{(A_h+J^h_{m-k-2}(\omega))(t_{m-k-1}-s)}\right\Vert_{L(H)}\nonumber\\
&\leq&\left\Vert e^{(A_h+J^h_{m-1}(\omega))\Delta t}\cdots e^{(A_h+J^h_{m-k-1}(\omega))\Delta t}\right\Vert_{L(H)}\left\Vert e^{(A_h+J^h_{m-k-2}(\omega))(t_{m-k-1}-s)}\right\Vert_{L(H)}\nonumber\\
&\leq& C,
\end{eqnarray}
for all $\omega\in\Omega$.
Substituting \eqref{revi6} in \eqref{revi5}, employing \lemref{lemma11}, inserting an appropriate power of $-A_h$, using \lemref{lemma8} (i) with $\gamma_1=0$ and $\gamma_2=\frac{\beta}{2}-\epsilon$, and \lemref{lemma3}  yields 
{\small
\begin{eqnarray}
\label{revi7}
\sqrt{V_{11}}&\leq& C\sum_{k=0}^{m-2}\int_{t_{m-k-2}}^{t_{m-k-1}}\left\Vert(G^h_{m-k-2})'(X^h(t_{m-k-2}))\left(e^{(A_h+J^h_{m-k-2})(s-t_{m-k-2})}-\mathbf{I}\right)X^h(t_{m-k-2})\right\Vert_{L^2(\Omega,H)}\nonumber\\
&\leq& C\sum_{k=0}^{m-2}\int_{t_{m-k-2}}^{t_{m-k-1}}\left\Vert\left(e^{(A_h+J^h_{m-k-2})(s-t_{m-k-2})}-\mathbf{I}\right)X^h(t_{m-k-2})\right\Vert_{L^2(\Omega,H)}\nonumber\\
&\leq& C\sum_{k=0}^{m-2}\int_{t_{m-k-2}}^{t_{m-k-1}}\left\Vert\left(e^{(A_h+J^h_{m-k-2})(s-t_{m-k-2})}-\mathbf{I}\right)(-A_h)^{-\frac{\beta}{2}+\epsilon}(-A_h)^{\frac{\beta}{2}-\epsilon} X^h(t_{m-k-2})\right\Vert_{L^2(\Omega,H)}\nonumber\\
&\leq& C \sum_{k=0}^{m-2}\int_{t_{m-k-2}}^{t_{m-k-1}}(s-t_{m-k-2})^{\frac{\beta}{2}-\epsilon}\left\Vert(-A_h)^{\frac{\beta}{2}-\epsilon} X^h(t_{m-k-2})\right\Vert_{L^2(\Omega,H)}ds\nonumber\\
&\leq& C \sum_{k=0}^{m-2}\int_{t_{m-k-2}}^{t_{m-k-1}}(s-t_{m-k-2})^{\frac{\beta}{2}-\epsilon}ds\leq C\Delta t^{\frac{\beta}{2}-\epsilon}.
\end{eqnarray}
}
Using the triangle inequality, \lemref{lemma11}, \lemref{lemma6}, \lemref{lemma3}, \lemref{lemma9} and \lemref{lemma8}, it holds that
{\small
\begin{eqnarray}
\label{revi8}
\sqrt{V_{12}}&\leq&\sum_{k=0}^{m-2}\int_{t_{m-k-2}}^{t_{m-k-1}}\left\Vert e^{(A_h+J^h_{m-1})\Delta t}\cdots e^{(A_h+J^h_{m-k-1})\Delta t}e^{(A_h+J^h_{m-k-2})(t_{m-k-1}-s)}\right.\nonumber\\
&&\left.(G^h_{m-k-2})'(X^h(t_{m-k-2}))\int_{t_{m-k-2}}^se^{(A_h+J^h_{m-k-2})(s-\sigma)}G^h_{m-k-2}(X^h(\sigma))d\sigma\right\Vert_{L^2(\Omega,H)}ds\nonumber\\
&\leq& \sum_{k=0}^{m-2}\int_{t_{m-k-2}}^{t_{m-k-1}}\left(\mathbb{E}\left[\left\Vert e^{(A_h+J^h_{m-1})\Delta t}\cdots e^{(A_h+J^h_{m-k-1})(t_{m-k-1}-s)}\right\Vert^2_{L(H)}\right.\right.\nonumber\\
&&\times \left\Vert e^{(A_h+J^h_{m-k-2})(t_{m-k-1}-s)}\right\Vert^2_{L(H)}\left.\left.\left\Vert\int_{t_{m-k-2}}^se^{(A_h+J^h_{m-k-2})(s-\sigma)}G^h_{m-k-2}(X^h(\sigma))d\sigma\right\Vert^2_{L(H)}\right]\right)^{1/2}ds\nonumber\\
&\leq& C\sum_{k=0}^{m-2}\int_{t_{m-k-2}}^{t_{m-k-1}}\left(\mathbb{E}\left[\int_{t_{m-k-2}}^s\left\Vert e^{(A_h+J^h_{m-k-2})(s-\sigma)}\right\Vert_{L(H)}\Vert G^h_{m-k-2}(X^h(\sigma))\Vert d\sigma\right]^2\right)^{1/2}ds\nonumber\\
&\leq& C\Delta t.
\end{eqnarray}
Since the expectation of the cross-product terms vanishes, Cauchy-Schwartz inequality yields
{\small
\begin{eqnarray*}
&&V_{13}\nonumber\\
&=&\mathbb{E}\left[\left\Vert\sum_{k=0}^{m-2}\int_{t_{m-k-2}}^{t_{m-k-1}}e^{(A_h+J^h_{m-1})\Delta t}\cdots e^{(A_h+J^h_{m-k-1})\Delta t}e^{(A_h+J^h_{m-k-2})(t_{m-k-1}-s)}\right.\right.\nonumber\\
&&\left.\left.(G^h_{m-k-2})'(X^h(t_{m-k-2}))\int_{t_{m-k-2}}^se^{(A_h+J^h_{m-k-2})(s-\sigma)}P_hdW(\sigma)ds\right\Vert^2\right]\nonumber\\
&=&\sum_{k=0}^{m-2}\mathbb{E}\left[\left\Vert\int_{t_{m-k-2}}^{t_{m-k-1}}\int_{t_{m-k-2}}^s e^{(A_h+J^h_{m-1})\Delta t}\cdots e^{(A_h+J^h_{m-k-1})\Delta t}e^{(A_h+J^h_{m-k-2})(t_{m-k-1}-s)}\right.\right.\nonumber\\
&&\left.\left. (G^h_{m-k-2})'(X^h(t_{m-k-2}))e^{(A_h+J^h_{m-k-2})(s-\sigma)}P_hdW(\sigma)ds\right\Vert^2\right]\nonumber\\
&\leq&\Delta t\sum_{k=0}^{m-2}\int_{t_{m-k-2}}^{t_{m-k-1}}\mathbb{E}\left[\left\Vert e^{(A_h+J^h_{m-1})\Delta t}\cdots e^{(A_h+J^h_{m-k-1})\Delta t}e^{(A_h+J^h_{m-k-2})(t_{m-k-1}-s)}\right\Vert^2_{L(H)}\right.\nonumber\\
&\times&\left.\left\Vert \int_{t_{m-k-2}}^s(G^h_{m-k-2})'(X^h(t_{m-k-2})) e^{(A_h+J^h_{m-k-2})(s-\sigma)}P_hdW(\sigma)\right\Vert^2\right] ds.
\end{eqnarray*}
}
Using again Cauchy-Schwartz inequality, it follows that
{\small
\begin{eqnarray*}
&&V_{13}\nonumber\\
&\leq&\Delta t\sum_{k=0}^{m-2}\int_{t_{m-k-2}}^{t_{m-k-1}}\left(\mathbb{E}\left\Vert e^{(A_h+J^h_{m-1})\Delta t}\cdots e^{(A_h+J^h_{m-k-1})\Delta t}e^{(A_h+J^h_{m-k-2})(t_{m-k-1}-s)}\right\Vert^4_{L(H)}\right)^{\frac{1}{2}}\nonumber\\
&\times&\left(\mathbb{E}\left\Vert\int_{t_{m-k-2}}^s (G^h_{m-k-2})'(X^h(t_{m-k-2})) e^{(A_h+J^h_{m-k-2})(s-\sigma)}P_hdW(\sigma)\right\Vert^4\right)^{\frac{1}{2}} ds.
\end{eqnarray*}
}
}
Using the Burkh\"{o}lder-Davis-Gundy inequality (\cite[Lemma 5.1]{Raphael}), \lemref{lemma9} and \lemref{lemma8}, it holds that
{\small
\begin{eqnarray*}
\label{revi9}
&&V_{13}\nonumber\\
&\leq&C\Delta t\sum_{k=0}^{m-2}\int_{t_{m-k-2}}^{t_{m-k-1}}\left(\left\Vert e^{(A_h+J^h_{m-1})\Delta t}\cdots e^{(A_h+J^h_{m-k-1})\Delta t}e^{(A_h+J^h_{m-k-2})(t_{m-k-1}-s)}\right\Vert^4_{L(H)}\right)^{\frac{1}{2}}\nonumber\\
&\times&\int_{t_{m-k-2}}^s\mathbb{E}\left\Vert (G^h_{m-k-2})'(X^h(t_{m-k-2})) e^{(A_h+J^h_{m-k-2})(s-\sigma)}P_hQ^{\frac{1}{2}}\right\Vert^2_{\mathcal{L}_2(H)}d\sigma ds\\
&\leq& C\Delta t\sum_{k=}^{m-2}\int_{t_{m-k-2}}^{t_{m-k-1}}\int_{t_{m-k-2}}^s\mathbb{E}\left[\left\Vert (G^h_{m-k-2})'(X^h(t_{m-k-2})) e^{(A_h+J^h_{m-k-2})(s-\sigma)}P_hQ^{\frac{1}{2}}\right\Vert^2_{\mathcal{L}_2(H)}\right]d\sigma ds.
\end{eqnarray*}
}
Using \lemref{lemma11}, inserting an appropriate power of $-A_h$, using \lemref{lemma10} \lemref{lemma8} (ii) with $\gamma_1=0$ (if $\beta\geq 1$) and \lemref{lemma8} (ii) with $\gamma_1=\frac{1-\beta}{2}$ (if $\beta\leq 1$) , it holds that
\begin{eqnarray}
\label{revi10}
&&\mathbb{E}\left[\left\Vert (G^h_{m-k-2})'(X^h(t_{m-k-2})) e^{(A_h+J^h_{m-k-2})(s-\sigma)}P_hQ^{\frac{1}{2}}\right\Vert^2_{\mathcal{L}_2(H)}\right]\nonumber\\
&\leq& \mathbb{E}\left[\left\Vert  e^{(A_h+J^h_{m-k-2})(s-\sigma)}(-A_h)^{\frac{1-\beta}{2}}(-A_h)^{\frac{\beta-1}{2}}P_hQ^{\frac{1}{2}}\right\Vert^2_{\mathcal{L}_2(H)}\right]\nonumber\\
&\leq& \mathbb{E}\left[\left\Vert  e^{(A_h+J^h_{m-k-2})(s-\sigma)}(-A_h)^{\frac{1-\beta}{2}}\right\Vert^2_{L(H)}\left\Vert(-A_h)^{\frac{\beta-1}{2}}P_hQ^{\frac{1}{2}}\right\Vert^2_{\mathcal{L}_2(H)}\right]\nonumber\\
&\leq& C(s-\sigma)^{\min(-1+\beta,0)}.
\end{eqnarray}
Substituting \eqref{revi10} in \eqref{revi9} yields 
\begin{eqnarray}
\label{revi11}
V_{13}&\leq &C\Delta t\sum_{k=0}^{m-2}\int_{t_{m-k-2}}^{t_{m-k-1}}\int_{t_{m-k-2}}^s(s-\sigma)^{\min(-1+\beta,0)}d\sigma ds\nonumber\\
&\leq &C\Delta t^{\min(1+\beta,2)}.
\end{eqnarray}
To estimate $\sqrt{V_{14}}$, we  note by using \lemref{lemma11} and \lemref{lemma4} that
\begin{eqnarray}
\label{revi12}
\Vert (-A_h)^{-\frac{\eta}{2}} R_{G^h_{m-k-2}}\Vert_{L^2(\Omega,H)}&\leq& C\left\Vert \Vert X^h(s)-X^h(t_{m-k-2})\Vert^2\right\Vert_{L^2(\Omega,H)}\nonumber\\
&\leq& C\Vert X^h(s)-X^h(t_{m-k-2})\Vert^2_{L^4(\Omega,H)}\nonumber\\
&\leq& C\Delta t^{\min(\beta,1)}.
\end{eqnarray}
Therefore the following estimate holds for $\sqrt{V_{14}}$
\begin{eqnarray}
\label{revi13}
\sqrt{V_{14}}\leq C\Delta t^{\min(\beta,1)}.
\end{eqnarray}
Substituting \eqref{revi13}, \eqref{revi13}, \eqref{revi8} and \eqref{revi7} in \eqref{revi3} yields 
\begin{eqnarray}
\label{revi14}
V_1\leq C\Delta t^{\beta-2\epsilon}.
\end{eqnarray}
Substituting \eqref{revi14} and \eqref{revi1a} in \eqref{revi1} yields 
\begin{eqnarray}
\label{revi15}
V\leq C\Delta t^{\beta-2\epsilon}+C\Delta t\sum_{k=0}^{m-1}\Vert X^h(t_k)-X^h_k\Vert^2_{L^2(\Omega,H)}.
\end{eqnarray}
Let us move to the estimate of $IV$. Applying the It\^{o}-isometry to \eqref{noise1} yields
\begin{eqnarray}
\label{quatre1}
IV&\leq& \int_{t_{m-1}}^{t_m}\left\Vert\left(e^{(A_h+J^h_{m-1})(t_m-s)}-e^{(A_h+J^h_{m-1})\Delta t}\right)P_hQ^{\frac{1}{2}}\right\Vert^2_{\mathcal{L}_2(H)}ds\\
&=&\int_{t_{m-1}}^{t_m}\left\Vert e^{(A_h+J^h_{m-1})(t_m-s)}\left(\mathbf{I}-e^{(A_h+J^h_{m-1})(s-t_{m-1})}\right)P_hQ^{\frac{1}{2}}\right\Vert^2_{\mathcal{L}_2(H)}ds\nonumber.
\end{eqnarray}
Inserting $(-A_h)^{\frac{1-\beta}{2}}(-A_h)^{\frac{\beta-1}{2}}$ in \eqref{quatre1},  using  \eqref{trace1} and \lemref{lemma10} yields
{\small
\begin{eqnarray}
\label{quatre2}
IV&\leq&\int_{t_{m-1}}^{t_m}\left\Vert e^{(A_h+J^h_{m-1})(t_m-s)}\left(\mathbf{I}-e^{(A_h+J^h_{m-1})(s-t_{m-1})}\right)(-A_h)^{\frac{1-\beta}{2}}\right\Vert^2_{L(H)}\nonumber\\
&&\times\left\Vert (-A_h)^{\frac{\beta-1}{2}}P_hQ^{\frac{1}{2}}\right\Vert^2_{\mathcal{L}_2(H)}ds\nonumber\\
&\leq&C\int_{t_{m-1}}^{t_m}\left\Vert e^{(A_h+J^h_{m-1})(t_m-s)}\left(\mathbf{I}-e^{(A_h+J^h_{m-1})(s-t_{m-1})}\right)(-A_h)^{\frac{1-\beta}{2}}\right\Vert^2_{L(H)}ds.
\end{eqnarray}
}
Inserting $(-A_h)^{\frac{1-\epsilon}{2}}(-A_h)^{\frac{\epsilon-1}{2}}$ in \eqref{quatre2} and using \lemref{lemma8} (ii)  with $\gamma_1=\frac{1-\epsilon}{2}$, \lemref{lemma8} (iv) with $\gamma_1=\frac{1-\epsilon}{2}$ and $\gamma_2=\frac{1-\beta}{2}$ (or \lemref{lemma8} with $\gamma_1=\frac{1-\epsilon}{2}$ and $\gamma_2=\frac{\beta-1}{2}$ if $\beta\in[1,2]$) yields
\begin{eqnarray}
\label{quatre3}
IV&\leq& C\int_{t_{m-1}}^{t_m}\left\Vert e^{(A_h+J^h_{m-1})(t_m-s)}(-A_h)^{\frac{1-\epsilon}{2}}\right\Vert^2_{L(H)}\nonumber\\
&&\times\left\Vert(-A_h)^{\frac{\epsilon-1}{2}}\left(\mathbf{I}-e^{(A_h+J^h_{m-1})(s-t_{m-1})}\right)(-A_h)^{\frac{1-\beta}{2}}\right\Vert^2_{L(H)}ds\nonumber\\
&\leq& C\int_{t_{m-1}}^{t_m}(t_m-s)^{-1+\epsilon}(s-t_{m-1})^{\beta-\epsilon}ds\nonumber\\
&\leq& C\Delta t^{\beta-\epsilon}\int_{t_{m-1}}^{t_m}(t_m-s)^{-1+\epsilon}ds\leq C\Delta t^{\beta}.
\end{eqnarray}
Let us now turn our attention to the estimate of $VI$. Since the expectation of the cross-product vanishes,  using Cauchy-Schwartz inequality, it follows from \eqref{noise2} that
\begin{eqnarray*}
VI&=&\mathbb{E}\left\Vert\sum_{k=0}^{m-2}\int_{t_{m-k-2}}^{t_{m-k-1}} e^{(A_h+J^h_{m-1})\Delta t}\cdots e^{(A_h+J^h_{m-k-1})\Delta t}e^{(A_h+J^h_{m-k-2})(t_{m-k-1}-s)}\right.\nonumber\\
&&\left.\left(\mathbf{I}-e^{(A_h+J^h_{m-k-2})(s-t_{m-k-2})}\right)P_hdW(s)\right\Vert^2\nonumber\\
&=&\sum_{k=0}^{m-2}\mathbb{E}\left\Vert e^{(A_h+J^h_{m-1})\Delta t}\cdots e^{(A_h+J^h_{m-k-1})\Delta t}\right.\nonumber\\
&&\left.\int_{t_{m-k-2}}^{t_{m-k-1}}e^{(A_h+J^h_{m-k-2})(t_{m-k-1}-s)}\left(\mathbf{I}-e^{(A_h+J^h_{m-k-2})(s-t_{m-k-2})}\right)P_hdW(s)\right\Vert^2\nonumber\\
&\leq&\sum_{k=0}^{m-2}\left(\mathbb{E}\left\Vert e^{(A_h+J^h_{m-1})\Delta t}\cdots e^{(A_h+J^h_{m-k-1})\Delta t}(-A_h)^{\frac{1-\epsilon}{2}}\right\Vert^4_{L(H)}\right)^{\frac{1}{2}}\nonumber\\
&\times&\left(\mathbb{E}\left\Vert\int_{t_{m-k-2}}^{t_{m-k-1}}(-A_h)^{\frac{-1+\epsilon}{2}} e^{(A_h+J^h_{m-k-2})(t_{m-k-1}-s)}\left(\mathbf{I}-e^{(A_h+J^h_{m-k-2})(s-t_{m-k-2})}\right)P_hdW(s)\right\Vert^4\right)^{\frac{1}{2}}.
\end{eqnarray*}
Using the Burkh\"{o}lder-Davis-Gundy inequality (\cite[Lemma 5.1]{Raphael}), it follows that
\begin{eqnarray}
\label{noise33}
&&VI\\
&\leq &C\sum_{k=0}^{m-2}\left(\mathbb{E}\left\Vert e^{(A_h+J^h_{m-1})\Delta t}\cdots e^{(A_h+J^h_{m-k-1})\Delta t}(-A_h)^{\frac{1-\epsilon}{2}}\right\Vert^4_{L(H)}\right)^{\frac{1}{2}}\nonumber\\
&&\int_{t_{m-k-2}}^{t_{m-k-1}}\mathbb{E}\left\Vert (-A_h)^{\frac{-1+\epsilon}{2}} e^{(A_h+J^h_{m-k-2})(t_{m-k-1}-s)}\left(\mathbf{I}-e^{(A_h+J^h_{m-k-2})(s-t_{m-k-2})}\right)P_hQ^{\frac{1}{2}}\right\Vert^2_{\mathcal{L}_2(H)}ds.\nonumber
\end{eqnarray}
Inserting $(-A_h)^{\frac{1-\beta}{2}}(-A_h)^{\frac{\beta-1}{2}}$ in \eqref{noise33}, using \eqref{trace1} and  \lemref{lemma10} yields
\begin{eqnarray}
\label{noise3}
&&VI\\
&\leq&C\sum_{k=0}^{m-2}\left(\mathbb{E}\left\Vert e^{(A_h+J^h_{m-1})\Delta t}\cdots e^{(A_h+J^h_{m-k-1})\Delta t}(-A_h)^{\frac{1-\epsilon}{2}}\right\Vert^4_{L(H)}\right)^{\frac{1}{2}}\nonumber\\
&&\int_{t_{m-k-2}}^{t_{m-k-1}}\mathbb{E}\left\Vert (-A_h)^{\frac{-1+\epsilon}{2}} e^{(A_h+J^h_{m-k-2})(t_{m-k-1}-s)}\left(\mathbf{I}-e^{(A_h+J^h_{m-k-2})(s-t_{m-k-2})}\right)(-A_h)^{\frac{1-\beta}{2}}\right\Vert^2_{L(H)}\nonumber\\
&&\left\Vert(-A_h)^{\frac{\beta-1}{2}}P_h Q^{\frac{1}{2}}\right\Vert^2_{\mathcal{L}_2(H)}ds\nonumber\\
&\leq& C\sum_{k=0}^{m-2}\left(\mathbb{E}\left\Vert e^{(A_h+J^h_{m-1})\Delta t}\cdots e^{(A_h+J^h_{m-k-1})\Delta t}(-A_h)^{\frac{1-\epsilon}{2}}\right\Vert^4_{L(H)}\right)^{\frac{1}{2}}\nonumber\\
&&\int_{t_{m-k-2}}^{t_{m-k-1}}\mathbb{E}\left\Vert (-A_h)^{\frac{-1+\epsilon}{2}} e^{(A_h+J^h_{m-k-2})(t_{m-k-1}-s)}\left(\mathbf{I}-e^{(A_h+J^h_{m-k-2})(s-t_{m-k-2})}\right)(-A_h)^{\frac{1-\beta}{2}}\right\Vert^2_{L(H)}ds.\nonumber
\end{eqnarray}
Using \lemref{lemma9} with $\nu=\frac{1-\epsilon}{2}$ yields
\begin{eqnarray}
\label{noise4}
VI
&\leq&  C \sum_{k=0}^{m-2}\int_{t_{m-k-2}}^{t_{m-k-1}}t_{k+1}^{-1+\epsilon}\mathbb{E}\left\Vert(-A_h)^{\frac{\epsilon-1}{2}} e^{(A_h+J^h_{m-k-2})(t_{m-k-1}-s)}\right.\nonumber\\
&&\left.\left(\mathbf{I}-e^{(A_h+J^h_{m-k-2})(s-t_{m-k-2})}\right)(-A_h)^{\frac{1-\beta}{2}}\right\Vert^2_{L(H)}ds.
\end{eqnarray}
Inserting $(-A_h)^{\frac{1-\epsilon}{2}}(-A_h)^{\frac{\epsilon-1}{2}}$  in \eqref{noise4} and using \lemref{lemma8}  (iii) with $\gamma_1=\gamma_2=\frac{1-\epsilon}{2}$, \lemref{lemma8} (iv) with $\gamma_1=\frac{1-\epsilon}{2}$ and $\gamma_2=\frac{1-\beta}{2}$ when $0\leq \beta\leq 1$ and \lemref{lemma8} (i) with $\gamma_1=\frac{1-\epsilon}{2}$ and $\gamma_2=\frac{1-\beta}{2}$ when $1\leq \beta\leq 2$ yields
\begin{eqnarray}
\label{noise5}
VI&\leq&C\sum_{k=0}^{m-2}\int_{t_{m-k-2}}^{t_{m-k-1}}t_{k+1}^{-1+\epsilon}\mathbb{E}\left[\left\Vert (-A_h)^{\frac{\epsilon-1}{2}}e^{(A_h+J^h_{m-k-2})(t_{m-k-1}-s)}(-A_h)^{\frac{1-\epsilon}{2}}\right\Vert^2_{L(H)}\right.\nonumber\\
&&\times\left.\left\Vert (-A_h)^{\frac{-1+\epsilon}{2}}\left(\mathbf{I}-e^{(A_h+J^h_{m-k-2})(s-t_{m-k-2})}\right)(-A_h)^{\frac{1-\beta}{2}}\right\Vert^2_{L(H)}\right]ds\nonumber\\
&\leq& C\sum_{k=1}^{m-2}\int_{t_{m-k-2}}^{t_{m-k-1}}t_k^{-1+\epsilon}(s-t_{m-k-1})^{\beta-\epsilon}ds\nonumber\\
&\leq& C\Delta t^{\beta-\epsilon}\sum_{k=1}^{m-2}t_k^{-1+\epsilon}\int_{t_{m-k-2}}^{t_{m-k-1}}ds=C\Delta t^{\beta-\epsilon}\sum_{k=1}^{m-2}t_k^{-1+\epsilon}\Delta t.
\end{eqnarray}
 Let us recall the following estimate 
 \begin{eqnarray}
 \label{noise7a}
 \sum_{k=1}^{m-2}t_k^{-1+\epsilon}\Delta t\leq C.
 \end{eqnarray}
Inserting \eqref{noise7a} in \eqref{noise5} yields
\begin{eqnarray}
\label{noise8}
VI\leq C\Delta t^{\beta-\epsilon}.
\end{eqnarray}
Substituting  \eqref{noise8}, \eqref{quatre3}, \eqref{revi15} and \eqref{decom2} in \eqref{decom1} and applying the discrete Gronwall lemma yields
\begin{eqnarray}
\Vert X^h(t_m)-X^h_m\Vert_{L^2(\Omega, H)}\leq C\Delta t^{\beta/2-\epsilon/2}\leq C\Delta t^{\beta/2-\epsilon}.
\end{eqnarray}
This completes the proof of \thmref{mainresult2}.
\section{Numerical simulations}
\label{experiment}
Here we provide three examples to sustain our theoretical results. The first example has exact solution.
The  reference solution  or ''the exact solution'' used in  the errors computation for our second and third example are taken to be the numerical solution with small time step.
In the legends of our graphs, we use the following notations
\begin{enumerate}
 \item  SERS  denotes the strong errors from our SERS scheme.
 \item  SETD1 denotes  the strong errors from  the stochastic exponential scheme \cite{Antonio1} given by \eqref{setd1}.
\end{enumerate}
The exponential matrix function $\varphi_1$ is computed by Krylov subspace technique with fixed dimension $m=10$ and tolerance $tol=10^{-6}$ \cite{kry,Advances,ATthesis}.
Note that we compute at  every time step  the action on  the exponential matrix function on a vector and not the whole exponential matrix function. Our code was implemented in Matlab 8.1.
Note that the initial solution is taken to be $X_0=0$ throughout our simulations, so optimal convergence order in time will depend only on the regularity of the noise.
\subsection{Additive noise with exact solution}
We first consider the following stochastic reaction diffusion equation with stiff reaction driven by additive noise in two dimensions with Neumann boundary conditions
\begin{eqnarray}
\label{example1}
dX(t)=[D \varDelta X(t)-100 X(t)]dt+dW(t), \quad X(0)=X_0, \quad t\in[0,T],
\end{eqnarray}
on the domain $\Lambda=[0,L_1]\times[0,L_2]$, $D=10^{-1}.$ 
A simple computation shows that the eigenfunctions   $\{e_{i,j}\}_{i,j\geq 0}=\{e^{(1)}_i\otimes e^{(2)}_j\}_{i,j\geq 0}$  
with the corresponding  eigenvalues $\{\lambda_{i,j}\}_{i,j\geq 0}=\{(\lambda_i^{(1)})^2+(\lambda_j^{(2)})^2\}$ of  $-\varDelta$ are given by
\begin{eqnarray}
\label{eigenfunction}
e^{(l)}_0(x)=\sqrt{\frac{1}{L_l}}, \quad e^{(l)}_i(x)=\sqrt{\frac{2}{L_l}}\cos(\lambda_i^{(l)}x), \quad \lambda_0^{(l)}=0, \quad \lambda_i^{(l)}=\frac{i\pi}{L_l},
\end{eqnarray}
where $l=1,2$, $x\in\Lambda$ and $i\in\mathbb{N}$.  In the abstract form \eqref{model} our linear operator $A$ is taken to be $A=D\varDelta$ and $F(X)=-100X$ 
which obviously satisfies Assumption \ref{assumption3} and Assumption \ref{assumption6b}.  We take $L_1=L_2=1$ and the triangulation $\mathcal{T}$ has been constructed from uniform Cartesian grid of sizes $\Delta x=\Delta y=1/100$.

We assume that the covariance operator $Q$ and $A$ have the same eigenfunctions. We take the following values for $\{q_{i,j}\}_{i+j>0}$ in the representation  \eqref{noise} 
\begin{eqnarray}
\label{eigennoise}
q_{i,j}=\dfrac{1}{(i^2+j^2)^{\beta+\delta}}, \quad 0\leq\beta\leq 2,\,\,\,\text{and}\quad \delta>0\,\,\,\text{small enough}.
\end{eqnarray}
We can easily prove that \assref{assumption6a} is fulfilled,  since
\begin{eqnarray*}
  \underset{(i,j) \in \mathbb{N}^{2}}{\sum}\lambda_{i,j}^{\beta-1}q_{i,j}<  \pi^{2}\underset{(i,j) \in \mathbb{N}^{2}}{\sum} \left( i^{2}+j^{2}\right)^{-(1+\delta)} <\infty, \;\;\;\;\; \quad  0\leq \beta \leq 2.
\end{eqnarray*}
To have trace class noise, it is enough to take  $\beta+\delta >1$.
 We take $\beta =1$ and  $\delta =0.001$. According to \thmref{mainresult2}, the order of  convergence 
in time should be close to $0.5$.
The exact solution of  \eqref{example1} is constructed in \cite{Antonio2}.
\begin{figure}[!ht]
 \includegraphics[width=0.75\textwidth]{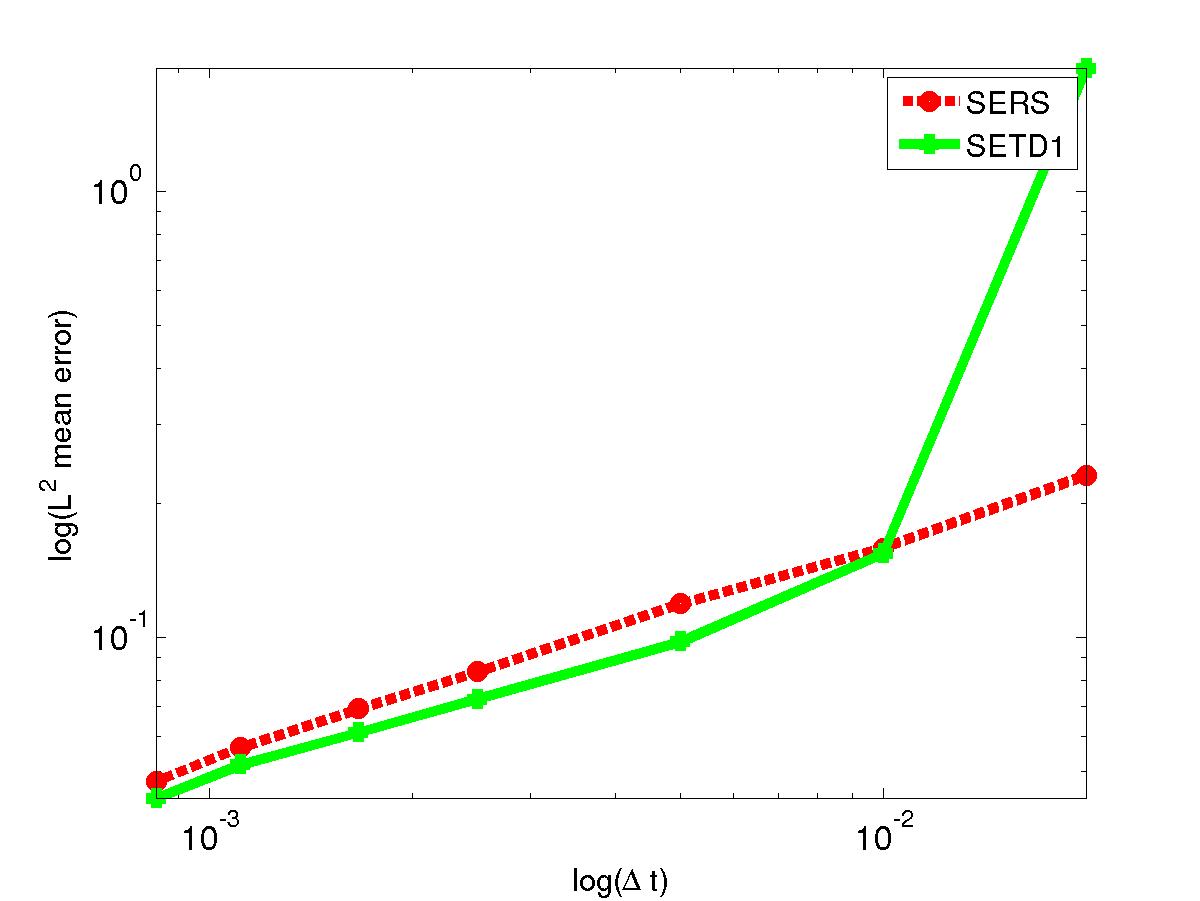}
  \caption{Strong convergence of SERS and SETD1 scheme, we can also observe that SETD1 is unstable for large time steps.
  The orders of convergence of the two methods are  $0.4971$ and $0.4980$ for SERS and SETD1 schemes respectively. The noise regularity parameter used is $\beta=1$ and  $50$
  samples are used in the errors computation.}
  \label{fig1}
\end{figure}
\figref{fig1} shows the  strong convergence of  SERS and SETD1 schemes. 
This figure also shows that SETD1 is unstable for large time steps.  We can  observe the good stability property of  the new  SERS scheme even for large time steps.
We can also  observe that the two schemes have the same order of accuracy. Indeed although SETD1 
seems to be more accurate, the two graphs become very close for small time step. The orders of convergence of the two methods are  $0.4971$ and $0.4980$ for SERS and SETD1 schemes respectively, 
which are very close to  theoretical results. 
Note that we only use  the stable part of  the data in the computation of the order of convergence  of SETD1 scheme.

\subsection{Additive noise without exact solution and with locally Lipschitz nonlinear function}
We  consider here  the following stochastic reaction diffusion equation  driven by additive noise in two dimensions with Neumann boundary conditions
\begin{eqnarray}
\label{example2}
dX(t)=[D \varDelta X(t)+ X(t)-X(t)^3]dt+dW(t), \quad X(0)=X_0,
\end{eqnarray}
on the domain $\Lambda=[0,L_1]\times[0,L_2]$, $D=10^{-2}$ and $t\in[0,T]$.   We take $L_1=L_2=1$ and the triangulation $\mathcal{T}$ has been constructed from uniform Cartesian grid of sizes $\Delta x=\Delta y=1/100$.
The  reference solution  or ''the exact solution'' using in  the errors computation is the numerical solution with the time step $\Delta t=1/2048$. 
The goal of this example is to prove that our novel scheme can be stable  and convergent for more complicated nonlinear function $F(X)=X-X^3$.
Although the existence and the uniqueness of  the solution of \eqref{example2} is well known \cite{Jentzen22,Prato}, the well-posedness of the numerical solution with our novel scheme
is not yet understood since  the nonlinear function is only locally  Lipschitz \cite{Prato,Jentzen22}. In our simulation, the  noise's representation \eqref{eigennoise}
is used with $\beta=1.2$  and $\delta=0.001$. The orders of convergence of the two methods are  $0.65$ and $0.62$ for SERS and SETD1 schemes respectively.
If our convergence theorem (\thmref{mainresult2}) was also valid for locally Lipschitz nonlinear function $F$, our convergence orders  should be then close to the expected order $0.6$.
We can also  observe that the two schemes have the same order of accuracy. Indeed although SETD1 
seems to be more accurate, the two graphs become very close for small time step. We can  also observe the good stability property of  the new  SERS scheme even for large time step.


\begin{figure}[!ht]
 \includegraphics[width=0.75\textwidth]{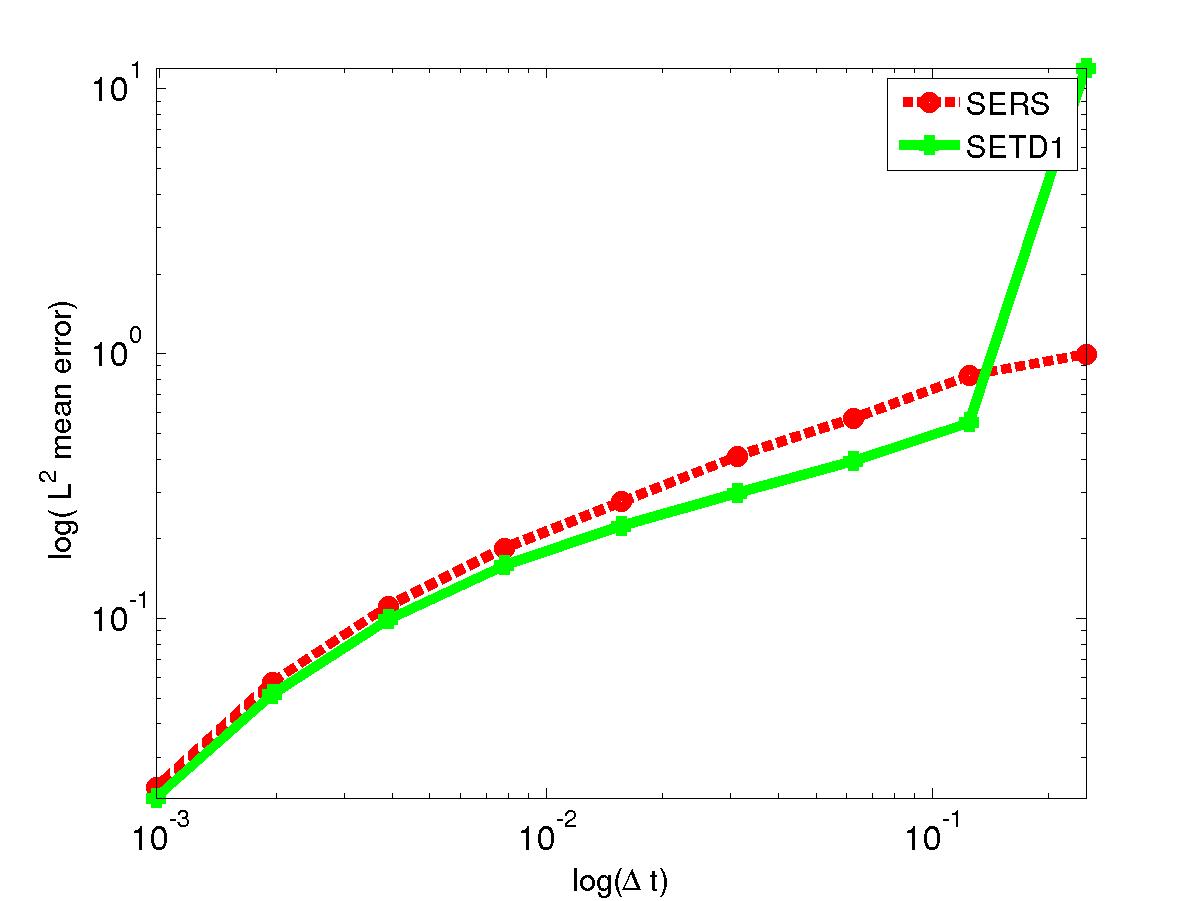}
 \caption{Strong convergence of SERS and SETD1 scheme can be observed for  nonlinear function $F(X)=X-X^3$.
  The orders of convergence of the two methods are  $ 0.65$ and $0.52$ for SERS and SETD1 schemes respectively. The noise regularity parameter used is $\beta=1.2$ and  $50$
  samples are used in the errors computation.}
  \label{fig3}
\end{figure}

\subsection{Multiplicative noise without exact solution}
As a more challenging example, we consider  the stochastic advection-diffusion-reaction SPDE with multiplicative noise in two dimensions on the domain $\Lambda=[0,1]\times[0,1]$. 
\begin{eqnarray}
\label{reactiondif1}
dX&=&\left[\nabla \cdot (\mathbf{D}\nabla X)-\nabla \cdot(\mathbf{q}X)-\dfrac{10 X}{ X +1}\right]dt+XdW.\\
\mathbf{D}&=&\left( \begin{array}{cc}
             10^{-2}&0\\
             0& 10^{-3}
             \end{array}\right)
\end{eqnarray}
with mixed Neumann-Dirichlet boundary conditions. The Dirichlet boundary condition is $X=1$ at $x=0$ 
and we use the homogeneous Neumann boundary conditions elsewhere. The Darcy velocity  $\mathbf{q}$ is obtained as in \cite{Antonio1}  and  to deal with high  P\'{e}clet flows 
we discretize in space using finite volume method (viewed as the finite element method as in \cite{Antonio3}) in rectangular grid of sizes $\Delta x=\Delta y= 1/110$. 
The  reference solution  or ''the exact solution'' using in  the errors computation is the numerical solution with the time step $\Delta t=1/2048$. 
Relatively small time steps are used to stabilize  the scheme SETD1.
The noise used is the same as in the first example with \eqref{eigennoise} and $\beta=1$ and $\delta=0.001$, corresponding to trace class noise.
Our linear operator $A$ is given by 
\begin{eqnarray}
 A=\nabla \cdot  \mathbf{D}\nabla (.) -  \nabla \cdot \mathbf{q}(.).
\end{eqnarray}
and the functions $f$ and $b$ are given by 
\begin{eqnarray}
\label{nemisk2}
f(x,u)=\dfrac{-10 u}{ u+1}, \quad b(x,u)=u,\quad \forall x\in\Lambda, \quad u\in\mathbb{R}.
\end{eqnarray}
Therefore, from \cite[Section 4]{Arnulf1} it follows that the operators $F$ and $B$ defined by \eqref{nemystskii} fulfil  obviously \assref{assumption3} and \assref{assumption4}.

\begin{figure}[!ht]
 \includegraphics[width=0.75\textwidth]{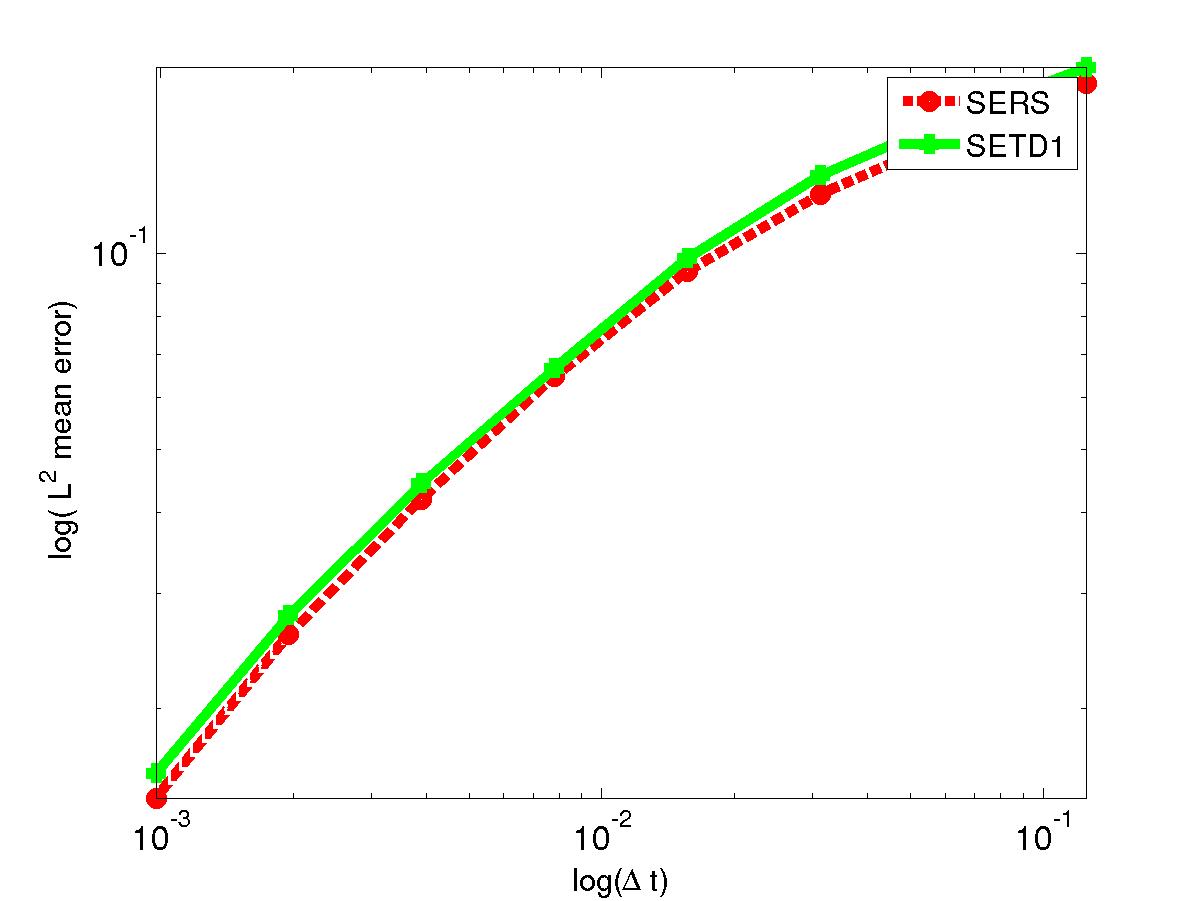}
 \caption{Strong convergence of SERS and SETD1 scheme can be observed.
  The orders of convergence of the two methods are  $ 0.5367$ and $0.5337$ for SERS and SETD1 schemes respectively. The noise regularity parameter used are $\beta=1$ and  $\delta=0.001$. Note that $50$
  samples have been used in the errors computation.}
  \label{fig2}
\end{figure}
\figref{fig2} shows the  strong convergence of  SERS scheme and SETD1 scheme presented in \cite{Antonio1}. 
We can also  observe that the two schemes have the same order of accuracy. Indeed although SERS
seems to be more accurate, the difference  between the two errors is small. 
The orders of convergence of the two methods are  $0.5367$ and $0.5337$ for SERS and SETD1 schemes respectively, 
which are very close to  $0.5$ (from  the theoretical results in \thmref{mainresult1}). 

\section{Concluding remark}
\label{concludremark}
In this work, we have analyzed the strong convergence of the exponential Rosenbrock-Euler method for a semilinear parabolic SPDE. The method is 
based on a continuous linearization of the problem at each time step. The linearisation technique consists of  
adding the Jacobian of the nonlinear function to a linear operator  while the nonlinear function is replaced by its reminder. 
The linear operator is assumed to be a generator of an  analytic semigroup. By \cite[Theorem 2.10, Page 176, Chapter 3]{Klaus}
there exists a constant $a>0$ such that $A+L$ generates an analytic semigroup for every $A$-bounded operator $L$ having $A$-bound $a_0<a$ (see \cite[Definition 2.1, Page 169, Chapter 3]{Klaus}).
As the nonlinear function $F$ is assumed to be Fr\'{e}chet differentiable with bounded derivative in our analysis,  
we can weaken that hypothesis on $F$ by  replacing \assref{assumption3} by the following  weaker assumption.
\begin{assumption}
\label{assumption7}
The nonlinear function $F : H\longrightarrow H$ is assumed to be Fr\'{e}chet differentiable with derivative relatively $A$-bounded with $A$-bound $a_0<a$, i.e there exist constants $ a_0\in[0,a)$ and $b\geq 0$ such that
\begin{eqnarray}
\label{ineqrev}
\Vert F'(u)v\Vert \leq a_0\Vert Av\Vert+b\Vert v\Vert, \quad u\in H,\, v\in \mathcal{D}(A).
\end{eqnarray}
\end{assumption}
Under \assref{assumption7}, for all $\omega\in\Omega$, $A_h+J^h_m(\omega)$ generates an analytic semigroup $S^h_m(\omega)(t):=e^{(A_h+J^h_m(\omega))t}$  (see \cite[Theorem 2.10, page 176, Chapter 3]{Klaus}) 
and therefore the numerical scheme \eqref{erem} is well posed. Under \assref{assumption7}, the convergence analysis of the numerical method \eqref{erem}
is not straightforward. This is due to the presence of the linear operator $A$ in the right hand side of \eqref{ineqrev}, which may produce some irregularities. 
This will be our interest for future work. Further investigations will be done also for locally Lipschitz nonlinear function $F$.

\section*{Acknowledgement}

This work was supported by the German Academic Exchange Service (DAAD) (DAAD-Project 57142917) and  the Robert Bosch Stiftung through the AIMS ARETE Chair programme (Grant No 11.5.8040.0033.0). Part of this
work was done when Antoine Tambue visited TU Chemnitz. The visit was supported by 
TWAS-DFG Cooperation Visits Programme. 

\end{document}